\documentclass{article}

\usepackage{xcolor}
\usepackage{xifthen}
\usepackage{tikz}
\usepackage{tikz-3dplot}
\usepackage{tikz-3dplot-circleofsphere}
\usepackage{pgfplots}
\usepackage{amsmath, amssymb}    

\usepackage{amsthm}
\usepackage{mathrsfs}        
\usepackage{amsfonts}
\usepackage{bbm}
\usepackage{wasysym}
\usepackage{mathtools}

 \usepackage{hyperref}
\hypersetup{colorlinks = true, linkcolor = blue, citecolor= magenta}

\usepackage[
  backend=biber,
  style=numeric,
  sortcites=true,
  defernumbers=true,
  giveninits=true,   
  eprint=false,
  url=false,         
  doi=false,         
  isbn=false         
]{biblatex}
\usepackage{graphicx} 

\usepackage{float} 
\usepackage[toc,page]{appendix}
\newtheorem*{theorem*}{Theorem}
\newtheorem{theorem}{Theorem}[section]

\newtheorem{lemma}[theorem]{Lemma}
\newtheorem{definition}[theorem]{Definition}
\newtheorem{notation}[theorem]{Notation}
\newtheorem{remark}[theorem]{Remark}
\newtheorem*{remark*}{Remark}
\newtheorem{proposition}[theorem]{Proposition}
\usepackage{enumitem}
\usepackage[section]{placeins}
\usepackage{float}
\usepackage[toc,page]{appendix}

\def\cZ{{\mathcal Z}}
\def\N{{\mathbb N}}    
\def\Z{{\mathbb Z}}    
\def\R{{\mathbb R}}   
   
\def\C{{\mathbb C}}
\def\S{{\mathbb S}}

  \def\cG{{\mathcal G}}     \def\cH{{\mathcal H}} \def\cN{{\mathcal N}}      \def\cU{{\mathcal U}}    \def\cP{{\mathcal P}}  \def\cE{{\mathcal E}}       \def\cL{{\mathcal L}}     \def\cZ{{\mathcal Z}}

\pgfplotsset{compat=1.18}

\title{Minimal time for null controllability of the parabolic spherical Baouendi-Grushin equation}
\author{Cyprien Tamekue\footnote{Department of Electrical and Systems Engineering, Washington University in St. Louis, St. Louis, MO, USA.
Email: cyprien@wustl.edu.}}

\date{}

\begin{document}

\maketitle

\begin{abstract}
   We study null controllability for the parabolic equation on \(\mathbb{S}^{2}\) endowed with its canonical almost-Riemannian structure. 
For a spherical crown $\omega=\{\alpha<x_3<\beta\}$, where $0\le \alpha<\beta\le1$, we prove the sharp minimal time formula
\(T_{\min}(\omega)=\ln(1/\sqrt{1-\alpha^{2}})\) for null controllability in $\omega$.
We also prove that, whenever the control region contains the equator, null controllability holds in every positive time. The proof combines two complementary tools. First, after Fourier decomposition with respect to the periodic variable, we establish observability estimates for a family of one-dimensional singular parabolic equations, with constants uniform with respect to the Fourier mode; the singularities at the poles are handled via a Hardy--Poincaré inequality. Second, for crowns away from the equator, we use the moment method to construct controls on the pole-touching crown $\alpha<x_3< 1$ from sharp weighted lower bounds on associated Legendre functions, and then pass to a general crown $\alpha<x_3<\beta$ by a cut-off argument on the full domain combined with the arbitrary-time controllability of crowns containing the equator. The result closes the large-time gap left in earlier work and gives the exact null-controllability threshold for the canonical almost-Riemannian heat equation on $\mathbb S^2$.
\end{abstract}

\tableofcontents

\section{Introduction}\label{s::introduction}

The controllability of parabolic equations with degeneracies has attracted sustained attention over the last two decades. A prototypical class is given by Baouendi--Grushin type operators, whose mixed elliptic--degenerate structure produces controllability phenomena that are absent in the uniformly parabolic setting; see, for instance, \cite{beauchard2014,beauchard2015,beauchard2020,koenig2017,duprez2020,darde2023null}. In particular, the location of the control region relative to the degeneracy set may impose a strictly positive lower bound on the time needed for null controllability.

In this paper, we determine the exact minimal time for null controllability of the parabolic equation associated with the spherical Baouendi--Grushin operator on the two-dimensional sphere $\S^2$, when the control acts on a one-sided spherical crown. More precisely, if the control region is
\[
\widetilde\omega=\{(x_1,x_2,x_3)\in \S^2:\alpha<x_3<\beta\},
\qquad 0\le \alpha<\beta\le 1,
\]
with $\alpha=\sin a$, $0\le a<\pi/2$, then the minimal time is exactly
\[
T_{\min}(\widetilde\omega)=\ln(1/\sqrt{1-\alpha^{2}})=\ln(1/\cos a).
\]
We also prove that null controllability holds in every positive time when the control region contains the equator. Thus, the paper gives a complete description of the threshold phenomenon for spherical crowns and shows that, in the one-sided case, the critical time depends only on the distance from the equator to the lower boundary of the crown. The operator arises from the canonical almost-Riemannian structure on $\S^2$, is hypoelliptic, and degenerates precisely on the equator. In earlier work~\cite{tamekue2022null}, the author proved the lower bound
$T_{\min}(\widetilde\omega)\ge \ln(1/\sqrt{1-\alpha^{2}})$ for crowns away from the equator, and established null controllability in sufficiently large time for symmetric double crowns. The main issue left open there was whether this lower bound is also sufficient for a one-sided crown. The present paper answers this question positively and therefore closes the gap left in~\cite{tamekue2022null}. The geometric setting and preliminary functional framework are developed in detail in~\cite[Chapter~2]{tamekue2023controllability}; we therefore recall here only the ingredients needed for the present analysis. We also refer to~\cite{agrachev2019,agrachev2008} for general background on almost-Riemannian structures. 

Let $\S^2:=\{p=(x_1,x_2,x_3)\in\R^3:\ x_1^2+x_2^2+x_3^2=1\}$.
We consider the almost-Riemannian structure generated by the two Killing vector fields $X_1$ and $X_2$, corresponding to rotations around the $x_1$- and $x_2$-axes, respectively. The associated intrinsic operator, defined with respect to the standard Riemannian volume $\mu$ on $\S^2$, is given by
\[
\mathcal L=\operatorname{div}_{\mu}\!\circ\nabla_{\mathrm{sR}}
=-X_1^{+}X_1-X_2^{+}X_2.
\]
This operator is hypoelliptic~\cite{hormander1967} and degenerates exactly on the equator $\cE=\{x_3=0\}$. In latitude--longitude coordinates
\[
\Phi:(x,y)\in\Big(-\frac{\pi}{2},\frac{\pi}{2}\Big)\times[0,2\pi)
\longmapsto (\cos x\cos y,\cos x\sin y,\sin x),
\]
the pullback $\Delta_{\mathrm{BG}}:=\Phi^*\mathcal L$ takes the form
\begin{equation}\label{eq:intro-BG}
\Delta_{\mathrm{BG}}
=\frac{1}{\cos x}\partial_x\!\big(\cos x\,\partial_x\big)+\tan^2x\,\partial_y^2 .
\end{equation}
Hence $\Delta_{\mathrm{BG}}$ is the spherical counterpart of the $2$D Baouendi--Grushin operator~\cite{beauchard2014}: it is elliptic away from $x=0$, degenerates along $\{x=0\}$, and exhibits coordinate singularities at the boundaries $x=\pm\pi/2$. Moreover, the coordinate representation on $\Omega$ of the Riemannian area $d\mu$ on $\S^2$ is given by
\begin{equation}\label{eq:measure on Omega}
    d\sigma(x,y)=\cos{x}dxdy,\qquad (x,y)\in\Omega.
\end{equation}

For uniformly parabolic heat equations on bounded domains, null controllability from any nonempty open set holds in arbitrarily small time; see, among many references, \cite{fattorini1971,fursikov1996,imanuilov1995,lebeau1995}. Degenerate parabolic equations behave differently. For the two-dimensional parabolic Baouendi--Grushin equation, Beauchard et al.~\cite{beauchard2014} showed that a strictly positive minimal time may be necessary, and that this threshold depends on the geometry of the control set. Minimal-time questions have also been studied for more general operators of the form $\cG_q=\partial_x^2+q(x)^2\partial_y^2$
on domains such as $(-L,L)\times(0,1)$, where $q$ vanishes on the degeneracy set; see, for instance, \cite{beauchard2015,beauchard2020}. However, these results do not directly apply to the present spherical model, because in the coordinates above the effective coefficient $q(x)=\tan x$ blows up at the poles $x=\pm\pi/2$. Handling this singular behavior is one of the main technical difficulties of the paper.

The proof of the sharp upper bound combines two complementary mechanisms.

\smallskip
\emph{(i) Uniform observability when the control set contains $\{x=0\}$ in its interior.}
Expanding in the periodic variable $y$ diagonalizes \eqref{eq:intro-BG} into the one-dimensional family
\[
\cL_n=\frac{1}{\cos x}\partial_x(\cos x\,\partial_x)-n^2\tan^2x,
\qquad n\in\Z.
\]
When the control region contains $\{x=0\}$, we prove Carleman estimates that are uniform with respect to the Fourier index $n$. A Hardy--Poincaré inequality is used to absorb the singular terms generated by the boundaries, which yields null controllability in every positive time for crowns containing the equator.

\smallskip
\emph{(ii) Moment method and cut-off argument when the control set is away or contains $\{x=0\}$ on its boundary.}
For the boundary-touching strip $(a,\pi/2)\times[0,2\pi)$, we construct controls by a moment method. The key input is a family of explicit lower bounds on the mass of associated Legendre functions,
\[
\int_a^{\frac{\pi}{2}} |v_{\ell,n}(x)|^2\cos x\,dx
\ge C_{\ell,n}\cos^{2n+2}a,
\qquad \ell\ge n\ge1,
\]
with constants $C_{\ell,n}>0$ tracked explicitly. These estimates compensate for the exponential concentration of high tangential frequencies near $\{x=0\}$ and imply null controllability for every $T>\ln(1/\cos a)$.
We then pass from the boundary-touching strip $(a,\pi/2)\times[0,2\pi)$ to a general strip $(a,b)\times[0,2\pi)$ by a cut-off argument on the full domain, combined with the arbitrary-time controllability of stripes containing $\{x=0\}$ in their interior.

Together with the lower bound from~\cite{tamekue2022null}, these arguments yield the exact minimal time for null controllability from a spherical crown.

\section{Main results}\label{ss::main results}
We consider the intrinsic parabolic spherical Baouendi-Grushin equation
\begin{equation}\label{eq::intrinsic parabolic BG equation}
	\begin{cases}
		\partial_t f - \mathcal L f = u \mathbf 1_{\widetilde{\omega}},&\qquad \text{in }(0,T)\times \mathbb S^2,\\
		f|_{t=0} = f_0,& \qquad\text{in }\mathbb S^2.
	\end{cases}
\end{equation}
Here $f=f(t,p)$ is the state, $f_0\in L^2(\mathbb S^2;\mu)$ is the initial datum, $u$ is the control, and $\mathbf 1_{\widetilde\omega}$ denotes the characteristic function of the control set $\widetilde\omega\subset\mathbb S^2$.

We say that \eqref{eq::intrinsic parabolic BG equation} is \textit{null controllable} from $\widetilde\omega$ in time $T>0$ if, for every $f_0\in L^2(\mathbb S^2;\mu)$, there exists $u\in L^2(0,T;L^2(\mathbb S^2;\mu))$ such that the corresponding solution satisfies $f(T,\cdot)=0$.

The following result, proved in~\cite[Theorem~1.2]{tamekue2022null}, is the starting point of the present work.
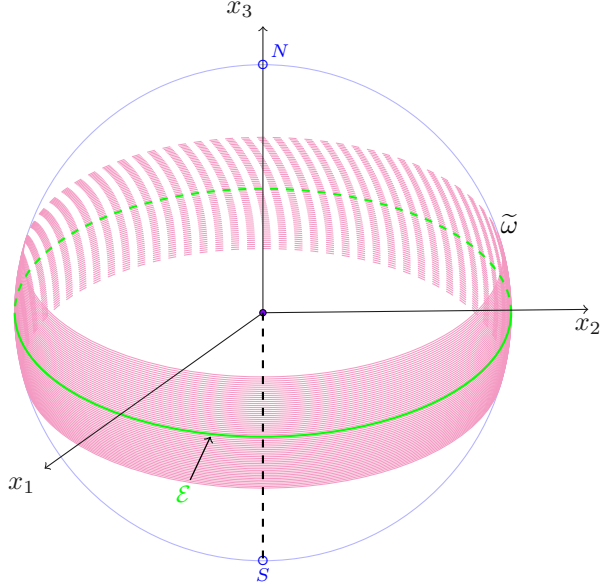
\begin{figure} 
	\centering
	\def\r{2.2}
	\tdplotsetmaincoords{60}{125}
	\pgfmathsetmacro{\rvec}{.8}
	\pgfmathsetmacro{\thetavec}{30}
	\pgfmathsetmacro{\phivec}{60}
	\definecolor{qqffff}{rgb}{0,1,1}
	\definecolor{qqqqff}{rgb}{0,0,1}
	\definecolor{ffqqqq}{rgb}{1,0,0}
	\definecolor{xfqqff}{rgb}{0.4980392156862745,0,1}
	\definecolor{ffqqff}{rgb}{1,0,1}
		\vspace{-0.4cm}
	\begin{tikzpicture}[tdplot_main_coords]
		\draw[tdplot_screen_coords,thin,blue!30] (0,0,0) circle (1.5*\r);
		\foreach \a in {-15,-14.5,...,15}
		{\tdplotCsDrawLatCircle[thin,magenta!50]{1.5*\r}{\a}}
		\tdplotCsDrawLatCircle[green,thick]{1.5*\r}{0}
		\begin{scope}[thin,black!90]
			\draw[->] (0,0,0) -- (2.3*\r,0,0) node[anchor=north east] {$x_{1}$};
			\draw[->] (0,0,0) -- (0,2.4*\r,1.8) node[anchor=north] {$x_{2}$};
			\draw[->] (0,0,0) -- (0,0,2*\r) node[anchor=south east] {$x_{3}$};
		\end{scope}
		\draw (0,4,2.4) node[above]{$\widetilde{\omega}$};
		\draw [->,line width=0.5pt] (4.2,0.8*\r,0) -- (3.1,0.6*\r,0);
		\draw (4.2,0.75*\r,0) node [below,scale=0.9,green] {$\cE$} ;
		\begin{scriptsize}
			\draw [fill=xfqqff] (0,0,0)  circle (1.2pt);
			\draw [fill=red] (0,0,1.73*\r)  node[blue,scale=1.2] {$\circ$};
			\draw [fill=red] (0,0,-1.73*\r)  node[blue,scale=1.2] {$\circ$};
			\draw (0,0,1.73*\r)node[blue, above right]{$N$};
			\draw[dashed] (0,0,-1.73*\r) node[blue, below]{$S$};
		\end{scriptsize}
		\draw [line width=0.8pt,dashed] (0,0,0)-- (0,0,-1.73*\r) ;
	\end{tikzpicture} 
	\caption{The control region $\widetilde{\omega}$ (in \textit{magenta}) is symmetric with respect to the equator $\cE$ (in \textit{green}), which is contained in its interior. The north and south poles are the points $N$ and $S$ respectively (in \textit{blue}).} \label{fig::Grushin sphere-control domain containing the equator} 
\end{figure}
\begin{theorem}[\cite{tamekue2022null}]
	\label{thm::main-intrinsic}
	Let $\widetilde{\omega} = \{ (x_1,x_2,x_3)\in \mathbb S^2\mid \alpha <|x_3|< \beta \}$ with $0<\alpha<\beta\le 1$. Then a positive minimal time is required for null controllability of \eqref{eq::intrinsic parabolic BG equation} from $\widetilde{\omega}$, namely
		\begin{equation*}
		T_{\min}(\widetilde{\omega}):=\inf\{T>0: \mbox{system \eqref{eq::intrinsic parabolic BG equation} is null controllable from $\widetilde{\omega}$ in time \;}T\}
	\end{equation*}	
	 satisfies $T_{\min}(\widetilde{\omega})\ge\ln (1 / \sqrt{1-\alpha^2})$. Moreover, there exists $T^{*}>0$ such that, for every $T\ge T^{*}$,  system \eqref{eq::intrinsic parabolic BG equation} is null controllable from $\widetilde{\omega}$ in time~$T$.
\end{theorem}

The lower-bound argument in~\cite[Section~5.1]{tamekue2022null} shows that \eqref{eq::intrinsic parabolic BG equation} is not null controllable from such a region when $T\le \ln (1 / \sqrt{1-\alpha^2})$.
The same obstruction already applies when the control set is a single crown away from the equator~\cite[Proposition~2.5.1]{tamekue2023controllability}. 

The main contribution below is to prove the matching upper bound.

\begin{theorem}\label{thm::intrinsic degenerate}
		Let $\widetilde{\omega} = \{ (x_1,x_2,x_3)\in \mathbb S^2\mid \alpha<x_3< \beta \}$ with $0\le\alpha<\beta\le 1$. Then the minimal time required for null controllability of the control system \eqref{eq::intrinsic parabolic BG equation} from $\widetilde{\omega}$ is $T_{\min}(\widetilde{\omega}) = \ln (1 / \sqrt{1-\alpha^2})$. More precisely,
  \begin{enumerate}
      \item For every $T>\ln (1 / \sqrt{1-\alpha^2})$, system \eqref{eq::intrinsic parabolic BG equation} is null controllable from $\widetilde{\omega}$ in time $T$;
      \item For every $T\le\ln (1 / \sqrt{1-\alpha^2})$, system \eqref{eq::intrinsic parabolic BG equation} is not null controllable from $\widetilde{\omega}$ in time $T$.
  \end{enumerate}
\end{theorem}

\begin{remark}
    Since $\alpha=0$ is allowed in Theorem~\ref{thm::intrinsic degenerate}, the system is null controllable in every positive time when the equator belongs to the boundary of the control set.
\end{remark}

The second theorem treats the case where the control region contains the equator in its interior.

\begin{theorem}\label{thm::intrinsic non-degenerate}
		Consider the control set $\widetilde{\omega} = \{ (x_1,x_2,x_3)\in \mathbb S^2\mid |x_3|< \beta \}$ with $0<\beta\le 1$. Then the control system \eqref{eq::intrinsic parabolic BG equation} is null controllable from $\widetilde{\omega}$ in an arbitrarily short time $T>0$.
\end{theorem}

\begin{remark}\label{rmk:important}
  The symmetry assumption in Theorem~\ref{thm::intrinsic non-degenerate} is used only to simplify the exposition and the figure. The proof extends to nonsymmetric control regions that contain a neighborhood of the equator.
\end{remark}

The proofs are carried out in latitude--longitude coordinates, as in~\cite{tamekue2022null}. In these coordinates, \eqref{eq::intrinsic parabolic BG equation} becomes
\begin{equation}\label{eq::spherical parabolic BG equation}
	\begin{cases}
		\displaystyle\partial_tf-\Delta_{\operatorname{BG}}f = u\mathbf 1_\omega,&\mbox{ in }(0,T)\times\Omega,\cr
		\displaystyle f|_{t=0} = f_{0},&\mbox{ in }\Omega,
	\end{cases}
\end{equation}
together with the boundary conditions induced by the latitude-longitude coordinates; see~\cite[Lemma~2.3]{tamekue2022null}. Here $\omega\subset\Omega$, $f=f(t,x,y)$ is the state, $f_0=f_0(x,y)$ is the initial datum, and $u=u(t,x,y)$ is the control.

The coordinate system \eqref{eq::spherical parabolic BG equation} is null controllable from $\omega\subset\Omega$ in time $T>0$ if, for every $f_0\in L^2(\Omega;\sigma)$, there exists $u\in L^2(0,T;L^2(\Omega;\sigma))$ such that $f(T,\cdot,\cdot)=0$.

Throughout the paper, let $0\le a<b\le\pi/2$ be such that $\alpha=\sin a$ and $\beta=\sin b$. We set
\begin{align}
    \Gamma:=\Gamma_{b}\times[0,2\pi),\qquad\mbox{with}\qquad\Gamma_{b} := (-b,b).\label{eq::degenerate control set in spherical coordinates}\\
    \omega:=\omega_{a,b}\times[0,2\pi),\qquad\mbox{with}\qquad\omega_{a,b} = (a,b).\label{eq::control set in spherical coordinates}
\end{align}

Then, Theorem~\ref{thm::intrinsic degenerate} is equivalent to the following.
\begin{theorem}\label{thm::main-spherical coordinates degenerate}
	Let $\omega$ be defined as in \eqref{eq::control set in spherical coordinates}. Then the minimal time required for null controllability of the control system \eqref{eq::spherical parabolic BG equation} from $\omega$ is $T_{\min}(\omega)=\ln (1 / \cos a)$. More precisely,
	  \begin{enumerate}
      \item For every $T>\ln (1 / \cos a)$, system \eqref{eq::spherical parabolic BG equation} is null controllable from $\omega$ in time $T$;
      \item For every $T\le\ln (1 / \cos a)$, system \eqref{eq::spherical parabolic BG equation} is not null controllable from $\omega$ in time $T$.
  \end{enumerate}
\end{theorem}
\begin{remark}
      For $a>0$, the non-controllability statement for every $T\le\ln(1/\cos a)$ was proved in~\cite[Section~5.1]{tamekue2022null}. Hence, the new insight of Theorem~\ref{thm::main-spherical coordinates degenerate} is the upper bound.
\end{remark}
Similarly, Theorem~\ref{thm::intrinsic non-degenerate} is equivalent to the following. 
\begin{theorem}\label{thm::intrinsic SBG non-degenerate}
	Let $\Gamma$ be defined as in \eqref{eq::degenerate control set in spherical coordinates}. Then, the control system \eqref{eq::spherical parabolic BG equation} is null controllable from $\Gamma$ in an arbitrarily short time $T>0$.
\end{theorem}

\subsection{Strategy for the proofs}\label{ss::trick}
It remains to prove Theorem~\ref{thm::intrinsic SBG non-degenerate} and the upper-bound part of Theorem~\ref{thm::main-spherical coordinates degenerate}. The lower-bound part follows from~\cite[Section~5.1]{tamekue2022null}.

\subsubsection{Proof strategy for Theorem~\ref{thm::main-spherical coordinates degenerate}}

The proof of the upper bound is based on Fourier decomposition in the angular variable. Writing
\begin{equation}
	f(t,x,y)=\sum_{n\in\Z}f_n(t,x)e^{iny},
\end{equation}
reduces \eqref{eq::spherical parabolic BG equation} to a family of one-dimensional parabolic control systems. Each equation is uniformly parabolic in the open interval $(-\pi/2,\pi/2)$ but carries singular boundary behavior at the poles.

We first treat the boundary-touching strip $(a,\pi/2)\times[0,2\pi)$. For this geometry, we use a moment method, in the spirit of Fattorini and Russell~\cite{fattorini1971}. In contrast with approaches (see, e.g.,~\cite{khodja2014minimal}) based on controls of separated form $u(t,x,y)=v(t)p(x,y)$, where $p$ is a fixed spatial profile supported in $\omega$, we use mode-dependent controls of the form
\begin{equation}
	u(t,x,y)=\sum_{n\in\Z}u_n(t,x)e^{iny},
	\qquad
	u_n(t,x)=\sum_{\ell\ge |n|}\alpha_{\ell,n}\,q_\ell^n(t)v_{\ell,n}(x)\mathbf 1_{\omega_{a,\pi/2}}(x).
\end{equation}
Here $(q_\ell^n)_\ell$ is biorthogonal in $L^2(0,T)$ to the family $(e^{-(T-t)\lambda_{\ell,n}})_\ell$, and the coefficients $\alpha_{\ell,n}$ are chosen to cancel each spectral component at time $T$. This type of distributed control was used in~\cite[Section~5]{allonsius2018} for parabolic equations associated with Sturm--Liouville operators, in~\cite{lagnese1983} for the exact controllability of the one-dimensional wave equation, and in~\cite{allonsius2021analysis} for the minimal time of the two-dimensional parabolic Baouendi--Grushin equation with a one-sided vertical control strip. A related construction appears in~\cite{benabdallah2020} for abstract uniformly parabolic systems.

The passage from the boundary-touching strip $(a,\pi/2)\times[0,2\pi)$ to a general strip $(a,b)\times[0,2\pi)$, with $0\le a<b<\pi/2$, is then obtained by a cut-off argument on the full domain. More precisely, we combine the moment-method controllability on $(a,\pi/2)\times[0,2\pi)$ with the arbitrary-time controllability from the symmetric strip $(-b,b)\times[0,2\pi)$, which contains the $y-$axis in its interior. A one-dimensional cut-off in the latitude variable localizes both controls inside $(a,b)\times[0,2\pi)$ while the commutator terms remain supported in the same strip.

\subsubsection{Proof strategy for Theorem~\ref{thm::intrinsic SBG non-degenerate}}
When the control set contains the equator (i.e., $\{x=0\}$ in latitude-longitude coordinates), we use the standard duality between null controllability and observability for the adjoint system; see, for instance, \cite{dolecki1977,lions1988}. The adjoint equation is
\begin{equation}\label{eq::homogeneous spherical parabolic BG equation}
	\begin{cases}
		\displaystyle\partial_tg-\Delta_{\operatorname{BG}}g =0,&\mbox{ in }(0,T)\times\Omega,\cr
		\displaystyle g|_{t=0} = g_{0},&\mbox{ in }\Omega,
	\end{cases}
\end{equation}
together with the boundary conditions induced by the latitude-longitude coordinates; see~\cite[Lemma~2.3]{tamekue2022null}.
\begin{definition}[Observability inequality]
	System~\eqref{eq::homogeneous spherical parabolic BG equation} is \textit{observable} in $\omega\subset\Omega$ in time $T>0$ if there exists $C=C(T,\omega)>0$ such that, for every $g_0\in L^2(\Omega;\sigma)$, the corresponding solution satisfies
\begin{equation}\label{eq::obervability inequality for spherical PBG equation}
	\int_\Omega |g(T,x,y)|^2\,d\sigma(x,y)
	\le C\int_0^T\int_\omega |g(t,x,y)|^2\,d\sigma(x,y)\,dt.
\end{equation}
\end{definition}

Thus, Theorem~\ref{thm::intrinsic SBG non-degenerate} is equivalent to the following observability statement.
\begin{theorem}\label{thm::main-spherical coordinates observability non-degenerate}
	Let $\Gamma$ be defined by~\eqref{eq::degenerate control set in spherical coordinates}. Then the adjoint system \eqref{eq::homogeneous spherical parabolic BG equation} is observable in $\Gamma$ for every $T>0$.
\end{theorem}

The proof proceeds by expanding $g$ in Fourier series with respect to $y$. The resulting family of one-dimensional adjoint equations is handled by Carleman estimates whose constants are uniform with respect to the Fourier index. Compared with~\cite{beauchard2014}, one must also treat the singularities at $x=\pm\pi/2$; this is achieved using the Hardy--Poincar\'e inequality of~\cite[Lemma~2.6]{tamekue2022null} together with suitable unitary transformations. The uniform one-dimensional observability inequalities are then summed over the Fourier modes by Bessel--Parseval's identity, yielding \eqref{eq::obervability inequality for spherical PBG equation}.

\subsection{Comments and open questions}\label{ss::comments}

The present results concern spherical crowns, for which the Fourier decomposition in the longitude variable is compatible with the geometry of the control set. More general geometries are substantially more delicate. For the flat two-dimensional Baouendi--Grushin equation and related generalized Grushin models, several works have already shown that the geometry of the control region may drastically affect null controllability; see, for instance, \cite{darde2023null,duprez2020,koenig2017,vanlaere2025non}. In particular, for certain control sets avoiding too much of the singular set, null controllability may fail in every positive time~\cite{koenig2017,vanlaere2025non}.

A natural continuation of the present work is therefore to study
\eqref{eq::intrinsic parabolic BG equation} for control regions on $\S^2$ that are not invariant under rotations in the longitude variable. For example, one may consider regions bounded by two meridians, as suggested in Figure~\ref{fig::Grushin sphere-control domain}. Such geometries no longer diagonalize the control problem mode by mode, and therefore require techniques beyond the moment construction used here. In this direction, the recent multidimensional work of~\cite{darde2026critical} suggests that a promising route may be to combine sharp observability estimates on structured observation regions with a Lebeau--Robbiano localization argument.
 
Another direction concerns compact almost-Riemannian surfaces. For the heat equation generated by the standard sub-Laplacian associated with the singular almost-Riemannian volume, null controllability fails when the control is contained in one connected component of $M\setminus\cZ$, where $\cZ$ is the singular set; see~\cite{boscain2013}. By contrast, the operator considered in the present paper is defined using a smooth volume form $\mu$, i.e., 
\[
\mathcal L=\operatorname{div}_\mu\circ\nabla_{\mathrm{sR}},
\] 
and therefore differs from the operator in~\cite{boscain2013}, where the divergence is taken with respect to the almost-Riemannian area, which blows up on $\cZ$.

Recent progress indicates that the geometric obstruction created by the singular set is robust beyond the spherical setting considered here. On the one hand, \cite{vanlaere2025non} proves lower bounds for the minimal time on certain complete almost-Riemannian manifolds under local Grushin-type assumptions. On the other hand, the recent paper of~\cite{vanlaere2026observability} shows, in singular Grushin settings, that the minimal time may depend not only on the observation set but also on the strength of the singularity, and even on the underlying measure in the almost Riemannian interpretation. It would therefore be very interesting to determine whether sharp upper bounds can be obtained for Baouendi--Grushin type heat equations on more general two-dimensional almost-Riemannian manifolds when the operator is defined with respect to a smooth reference measure.
 
 More generally, the recent results~\cite{darde2026critical,vanlaere2025non,vanlaere2026observability} suggest two complementary questions in the spherical setting: first, whether the minimal time remains governed by the same Agmon-type quantity for more general control geometries; second, whether one can recover matching upper bounds by methods that no longer rely on Fourier diagonalization, such as Carleman-based localization or Lebeau--Robbiano strategies.

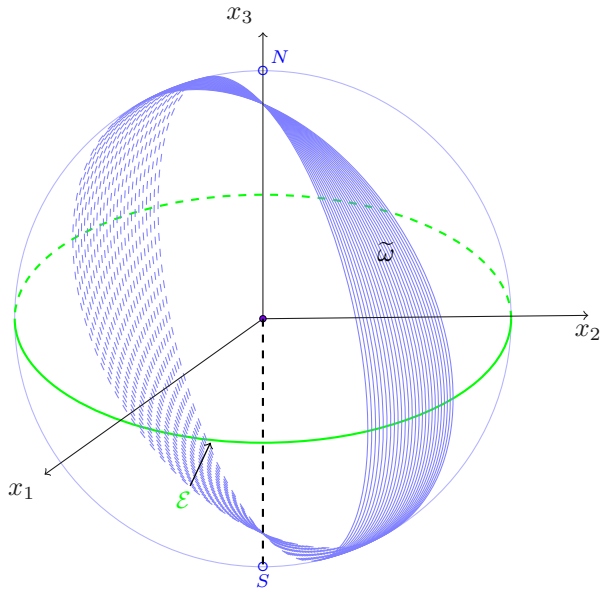
\begin{figure}
	\centering
	\def\r{2.2}
       \def\e{80}
	\tdplotsetmaincoords{60}{125}
	\pgfmathsetmacro{\rvec}{.8}
	\pgfmathsetmacro{\thetavec}{30}
	\pgfmathsetmacro{\phivec}{60}
	\definecolor{qqffff}{rgb}{0,1,1}
	\definecolor{qqqqff}{rgb}{0,0,1}
	\definecolor{ffqqqq}{rgb}{1,0,0}
	\definecolor{xfqqff}{rgb}{0.4980392156862745,0,1}
	\definecolor{ffqqff}{rgb}{1,0,1}
		\vspace{-0.4cm}
	\begin{tikzpicture}[tdplot_main_coords]
		\draw[tdplot_screen_coords,thin,blue!30] (0,0,0) circle (1.5*\r);
		\tdplotCsDrawLatCircle[green,thick]{1.5*\r}{0}
		 \foreach \a in {150,151,...,175}
		 {\tdplotCsDrawLonCircle[thin,blue!50]{1.5*\r}{\a}}
		\begin{scope}[thin,black!90]
			\draw[->] (0,0,0) -- (2.3*\r,0,0) node[anchor=north east] {$x_{1}$};
			\draw[->] (0,0,0) -- (0,2.4*\r,1.8) node[anchor=north] {$x_{2}$};
			\draw[->] (0,0,0) -- (0,0,2*\r) node[anchor=south east] {$x_{3}$};
		\end{scope}
		\draw (0,2,1.4) node[above]{$\widetilde{\omega}$};
		\draw [->,line width=0.5pt] (4.2,0.8*\r,0) -- (3.1,0.6*\r,0);
		\draw (4.2,0.75*\r,0) node [below,scale=0.9,green] {$\cE$} ;
		\begin{scriptsize}
			\draw [fill=xfqqff] (0,0,0)  circle (1.2pt);
			\draw [fill=red] (0,0,1.73*\r)  node[blue,scale=1.2] {$\circ$};
			\draw [fill=red] (0,0,-1.73*\r)  node[blue,scale=1.2] {$\circ$};
			\draw (0,0,1.73*\r)node[blue, above right]{$N$};
			\draw[dashed] (0,0,-1.73*\r) node[blue, below]{$S$};
		\end{scriptsize}
		\draw [line width=0.8pt,dashed] (0,0,0)-- (0,0,-1.73*\r) ;
	\end{tikzpicture} 
	\caption{The control region $\widetilde{\omega}$ (in \textit{blue}), the equator $\cE$ (in \textit{green}), and the north and south poles (in \textit{blue}).} \label{fig::Grushin sphere-control domain} 
\end{figure}

\subsection{Structure of the paper}\label{ss::structure}

The rest of the paper is organized as follows. Section~\ref{s::WP-FE-UNC} recalls the well-posedness framework, the Fourier decomposition, and the reduction of the controllability problem to uniform observability and controllability estimates for a family of one-dimensional parabolic equations.

Section~\ref{s:equator inside the control set} proves Theorem~\ref{thm::main-spherical coordinates observability non-degenerate}, and hence Theorem~\ref{thm::intrinsic SBG non-degenerate}, when the control set contains the equator. The key point is a uniform Carleman estimate for the one-dimensional adjoint equations, together with a Hardy--Poincaré inequality that handles the singularities generated by the spherical coordinates at the poles.

Section~\ref{s::minimal time} proves the upper-bound part of Theorem~\ref{thm::main-spherical coordinates degenerate}. Section~\ref{ss::useful estimates} establishes the weighted lower bounds for associated Legendre functions needed for the moment construction on the boundary-touching strip $(a,\pi/2)\times[0,2\pi)$, Section~\ref{ss::MM} applies the moment method to obtain controllability in every time $T>\ln(1/\cos a)$ for that geometry, and Section~\ref{ss:cut-off argument} extends the result to a general strip $(a,b)\times[0,2\pi)$ by a cut-off argument on the full domain.

\section{Well-posedness, Fourier expansion, uniform null 
controllability and uniform observability}\label{s::WP-FE-UNC}

Although the latitude-longitude coordinate representation of the control system~\eqref{eq::spherical parabolic BG equation} is naturally posed on the real Hilbert space $\operatorname{H}_\sigma:=L^2(\Omega;\sigma)$, it is convenient in what follows to work on its complexification, still denoted by $\operatorname{H}_\sigma$, in order to use the complex Fourier basis $(e^{iny})_{n\in\Z}$ in the angular variable $y$.

\subsection{Well-posedness}\label{ss::WP}

The well-posedness of \eqref{eq::intrinsic parabolic BG equation} and~\eqref{eq::spherical parabolic BG equation} was established
in~\cite[Section~2]{tamekue2022null}. For completeness, we recall the
coordinate formulation needed below, starting with the spectral
description of the spherical Baouendi--Grushin operator.

\begin{proposition}\label{pro::self-adjoiness of spherical Baoeundi-Grushin operator}
The operator $(-\Delta_{\operatorname{BG}},D(\Delta_{\operatorname{BG}}))$
is a nonnegative, densely-defined, hypoelliptic, symmetric, self-adjoint
operator on $L^2(\Omega;\sigma)$ and has a compact resolvent. Moreover,
\begin{equation}\label{eq::eigenvalues and eigenfunctions of the sub-Laplacian}
-\Delta_{\operatorname{BG}}W_{\ell,n}
=
\lambda_{\ell,n}W_{\ell,n},
\qquad
\lambda_{\ell,n}:=\ell(\ell+1)-n^2,
\qquad
\forall\, |n|\le \ell\in\N,
\end{equation}
where $\{\lambda_{\ell,n}\}_{\ell\in\N,\,-\ell\le n\le\ell}\subset\R_+$ are
the eigenvalues of $-\Delta_{\operatorname{BG}}$, and
$\{W_{\ell,n}\}_{\ell\in\N,\,-\ell\le n\le\ell}\subset D(\Delta_{\operatorname{BG}})$
the corresponding eigenfunctions are spherical harmonics defined for every
$(x,y)\in[-\pi/2,\pi/2]\times[0,2\pi)$ by
\begin{equation}\label{eq::spherical harmonics}
W_{\ell,n}(x,y)=v_{\ell,n}(x)e^{iny},
\qquad
v_{\ell,n}(x):=
\sqrt{\frac{2\ell+1}{4\pi}\frac{(\ell-n)!}{(\ell+n)!}}\,
P_\ell^n(\sin x).
\end{equation}
Here $P_\ell^n$ denotes the associated Legendre function of the first kind.
\end{proposition}

Let $\operatorname{H}_\sigma:=L^2(\Omega;\sigma)$, and denote its inner
product and norm by
$\langle\cdot,\cdot\rangle_{\operatorname{H}_\sigma}$ and
$\|\cdot\|_{\operatorname{H}_\sigma}$. By
Proposition~\ref{pro::self-adjoiness of spherical Baoeundi-Grushin operator},
$(\Delta_{\operatorname{BG}},D(\Delta_{\operatorname{BG}}))$ generates a
strongly continuous contraction semigroup
$(e^{t\Delta_{\operatorname{BG}}})_{t\ge0}$ on $\operatorname{H}_\sigma$.
For $f_0\in \operatorname{H}_\sigma$, $e^{t\Delta_{\operatorname{BG}}}f_0$
is the value at time $t$ of the homogeneous solution of
\eqref{eq::homogeneous spherical parabolic BG equation}; for $t>0$, it is
smooth in the interior and admits the expansion
\begin{equation}\label{eq::spherical Grushin semigroup}
e^{t\Delta_{\operatorname{BG}}}f_0
=
\sum_{|n|\le \ell\in\N}
e^{-t\lambda_{\ell,n}}
\langle f_0,W_{\ell,n}\rangle_{\operatorname{H}_\sigma}
W_{\ell,n}.
\end{equation}

We next introduce the form domain associated with
$-\Delta_{\operatorname{BG}}$,
\begin{equation}\label{eq::energy-space}
\mathcal V_\sigma
:=
D\!\big((-\Delta_{\operatorname{BG}})^{1/2}\big),
\end{equation}
endowed with the norm
\begin{equation}\label{eq:graph norm of the form}
    \|f\|_{\mathcal V_\sigma}^2
:=
\|f\|_{\operatorname{H}_\sigma}^2
+
\big\|(-\Delta_{\operatorname{BG}})^{1/2}f\big\|_{\operatorname{H}_\sigma}^2.
\end{equation}

By the spectral theorem, since $-\Delta_{\operatorname{BG}}$ is a nonnegative
self-adjoint operator with compact resolvent, its form domain
$\mathcal V_\sigma$ is characterized by
\[
f\in\mathcal V_\sigma
\quad\Longleftrightarrow\quad
\sum_{|n|\le \ell\in\N}\lambda_{\ell,n}
\big|\langle f,W_{\ell,n}\rangle_{\operatorname H_\sigma}\big|^2<\infty,
\]
and the graph norm~\eqref{eq:graph norm of the form} is equivalently given by
\begin{equation}\label{eq::energy-space-spectral}
\|f\|_{\mathcal V_\sigma}^2
=
\sum_{|n|\le \ell\in\N}(1+\lambda_{\ell,n})
\big|\langle f,W_{\ell,n}\rangle_{\operatorname{H}_\sigma}\big|^2,
\qquad f\in \mathcal V_\sigma.
\end{equation}
See, for instance, Davies~\cite[Chapter~4, Sections~4.3--4.4]{davies1995}. Moreover, for every smooth function $f$,
\begin{equation}\label{eq::quadratic-form-BG}
\big\|(-\Delta_{\operatorname{BG}})^{1/2}f\big\|_{\operatorname{H}_\sigma}^2
=
\int_\Omega
\Big(
|\partial_x f(x,y)|^2+\tan^2x\,|\partial_y f(x,y)|^2
\Big)\,d\sigma(x,y).
\end{equation}
In particular, the space $\mathcal V_\sigma$ is the natural energy
space associated with $-\Delta_{\operatorname{BG}}$, in the sense of the classical
variational framework of Lions~\cite[Chapitre~II]{lions1961}.

The following well-posedness result is standard; see, for instance,
\cite[Chapter~4]{pazy2012}.
\begin{proposition}\label{pro::well-posedness}
Given $T>0$, $f_0\in\operatorname{H}_\sigma$, and
$v:=\mathbf 1_\omega u\in L^2(0,T;\operatorname{H}_\sigma)$, there exists a
unique solution $f\in C([0,T];\operatorname{H}_\sigma)\cap L^2(0,T;\mathcal V_\sigma)$ of system~\eqref{eq::spherical parabolic BG equation}, satisfying
\begin{align}
f(t)
&=
e^{t\Delta_{\operatorname{BG}}}f_0
+\int_0^t e^{(t-s)\Delta_{\operatorname{BG}}}v(s)\,ds,
\qquad \forall t\in[0,T], \label{eq::solution equation in coordinates}\\
\|f(t)\|_{\operatorname{H}_\sigma}
&\le
\|f_0\|_{\operatorname{H}_\sigma}
+\int_0^t\|v(s)\|_{\operatorname{H}_\sigma}\,ds,
\qquad \forall t\in[0,T]. \label{e::estimation of solution}
\end{align}
Moreover, there exists a constant $C_T>0$ such that
\begin{equation}\label{eq::energy-estimate}
\|f\|_{C([0,T];\operatorname{H}_\sigma)}
+
\|f\|_{L^2(0,T;\mathcal V_\sigma)}
\le
C_T\Big(
\|f_0\|_{\operatorname{H}_\sigma}
+
\|v\|_{L^2(0,T;\operatorname{H}_\sigma)}
\Big).
\end{equation}
\end{proposition}

\subsection{Fourier expansion}\label{ss::FE}

Expanding in the Fourier basis $(e^{iny})_{n\in\Z}$ of $L^2((0,2\pi);dy)$ yields the orthogonal decomposition
$\operatorname{H}_\sigma=\bigoplus_{n\in\Z}^{\perp}\cH_n$, where $\cH_n\cong L^2((-\pi/2,\pi/2);\cos x\,dx)$. Consequently,
\begin{align}
e^{t\Delta_{\operatorname{BG}}} = \bigoplus_{n\in\Z}^{\perp}e^{t\cL_n},\qquad\forall t\ge 0,\label{eq::semigroup diagonalisation}\\
    D(\cL_{n}) = \left\{v\in\cH_n\mid\cL_{n}v\in\cH_n, v(\pm\pi/2)\in\C, v'(\pm\pi/2)\in\C\right\},\label{eq::domain operator L_n}\\
    \cL_{n}v = \frac{1}{\cos x}(\cos xv')'-n^{2}\tan^{2}xv,\qquad\quad v\in D(\cL_{n}).\label{eq::operator L_n}
\end{align}

The boundary conditions in $D(\cL_n)$ are those induced by the latitude-longitude coordinates. Moreover, for $\ell\in\N$, $n\in\Z$, and $|n|\le\ell$, the functions
\begin{equation}\label{eq::eigenfunctions of L_n}
	v_{\ell,n}(x) = \sqrt{\frac{2\ell+1}{2}\frac{(\ell-n)!}{(\ell+n)!}}P_\ell^n(\sin x),\hspace{1cm}\forall x\in [-\pi/2,\pi/2],
\end{equation}
form a complete orthonormal set of the Hilbert space $\cH_{n}$ \cite[p. 512]{courant1953}, with $P_{\ell}^{n}$ being the associated Legendre function of the first kind. Moreover, each $v_{\ell,n}$ lies in $D(\cL_{n})$ and we have $$-\cL_{n}v_{\ell,n} = \lambda_{\ell,n}v_{\ell,n},\qquad\qquad\lambda_{\ell,n} = \ell(\ell+1)-n^2.$$ 

Since the solution of adjoint system~\eqref{eq::homogeneous spherical parabolic BG equation}, $(t,x,y)\mapsto g(t,x,y)$ belongs to $C([0,T];\operatorname{H_{\sigma}})$, the function $y\mapsto g(t,x,y)$ belongs to $ L^2((0,2\pi);dy)$ for a.e. $(t,x)\in(0,T)\times(-\pi/2 ,\pi/2)$. It follows that \eqref{eq::homogeneous spherical parabolic BG equation} is equivalent to 
\begin{equation}\label{eq::one-dim parabolic equations}
	\begin{cases}
		\partial_tg_n-\cL_ng_n = 0,&\mbox{ in } (0,T)\times(-\pi/2,\pi/2),\cr
		g_n|_{t=0} = g_{0,n},&\mbox{ in } (-\pi/2,\pi/2).
	\end{cases}
\end{equation}
Here, the $n$-th Fourier component $g_n$ is given by
\begin{equation}\label{eq::fourier coefficient of g}
	g_n(t,x) = \int_{0}^{2\pi}g(t,x,y)e^{iny}dy,\qquad(t,x)\in(0,T)\times(-\pi/2,\pi/2),
\end{equation}
and $g_{0,n}$ is the n-th Fourier coefficients of the initial state $g_0\in\operatorname{H_{\sigma}}$ in system~\eqref{eq::homogeneous spherical parabolic BG equation}.

Due to \eqref{eq::semigroup diagonalisation}, one deduces that for every $n\in\Z$, the operator $(\cL_{n}, D(\cL_{n}))$ is the generator of the one-parameter strongly continuous semigroup of contraction $(e^{t\cL_{n}})_{t\ge 0}$ on $\cH_{n}$, which is defined for every $t\ge 0$ as follows: given $g_{0,n}\in\cH_n$, $e^{t\cL_n}g_{0,n}$ is the unique solution at time $t$ of the homogeneous equation of \eqref{eq::one-dim parabolic equations}, which is $C^\infty$ on $]0,+\infty[\times(-\pi/2,\pi/2)$ and given by
	\begin{equation}\label{eq::1D spherical Grushin semigroup}
		e^{t\cL_n}g_{0,n} = \sum\limits_{\ell\in\N}e^{-t\lambda_{\ell,n}}\langle g_{0,n},v_{\ell,n}\rangle_{\cH_n}v_{\ell,n}.
	\end{equation}
 
Let us expand the solution $f\in C([0,T];\operatorname{H_{\sigma}})$ of control system \eqref{eq::spherical parabolic BG equation} in terms of Fourier series as 
\begin{equation}\label{eq::coef of f}
			f(t,x,y) = \sum\limits_{n\in\Z}f_n(t,x)e^{iny},\qquad(t,x)\in(0,T)\times(-\pi/2,\pi/2),
		\end{equation}
 and the control $u\in L^2(0,T;\operatorname{H}_\sigma)$ as
\begin{equation}\label{eq::representation of u in Fourier series}
    u(t,x,y)= \sum\limits_{n\in\Z}u_n(t,x)e^{iny},\qquad(t,x)\in(0,T)\times(-\pi/2,\pi/2).
\end{equation}
Let $\omega:=\omega_{x}\times[0,2\pi)$ where $\omega_{x}\subset(-\pi/2,\pi/2)$. Then the control system \eqref{eq::spherical parabolic BG equation} is equivalent to the following family of $1$D (non-degenerate and singular at $\pm\pi/2$) control system
\begin{equation}\label{eq::one-dim control system}
	\begin{cases}
		\partial_tf_n-\cL_nf_n = u_n\mathbf 1_{\omega_x},&\mbox{ in } (0,T)\times(-\pi/2,\pi/2),\cr
		f_n|_{t=0} = f_{0,n},&\mbox{ in } (-\pi/2,\pi/2),
	\end{cases}
\end{equation}
where $f_{0,n}$ is the n-th Fourier coefficients of the initial state $f_0\in\operatorname{H_{\sigma}}$ in system~\eqref{eq::spherical parabolic BG equation}.

The following result then follows immediately, \cite[Chapter~4]{pazy2012}.
\begin{proposition}\label{pro::existence of solution of 1D parabolic equation}
	Let $T>0$. For every $n\in\Z$, the n-th Fourier component $f_n\in  C([0,T];\cH_{n})$ is the solution of \eqref{eq::one-dim control system} 
	which is represented for all $0\le t\le T$ as
	\begin{equation}
		f_n(t) = e^{t\cL_n}f_{0,n}+\int_{0}^{t}e^{(t-s)\cL_n}u_n(s)\mathbf 1_{\omega_x}ds.
	\end{equation}
\end{proposition}

\begin{remark}[Decay rate]
    By using \eqref{eq::1D spherical Grushin semigroup}, we obtain that the Fourier component $g_n$ satisfies the following dissipation rate
\begin{equation}\label{eq::dissipation rate}
	\|g_n(T,\cdot)\|_{\cH_{n}}\le e^{-|n|(T-t)}\|g_n(t,\cdot)\|_{\cH_{n}},\qquad\qquad\forall t\in (0,T).
\end{equation}
\end{remark}
\begin{notation}
	In what follows, we shall assume $n\in\N$ to simplify the notation. The same considerations hold for $n\in\Z_{-}$ by replacing $n$ with $|n|$.
\end{notation}

\subsection{Uniform null controllability and uniform observability}\label{ss::UNC}

We illustrate in this section how the proof of null controllability for the control system \eqref{eq::spherical parabolic BG equation} (resp. observability of the adjoint system \eqref{eq::homogeneous spherical parabolic BG equation}) reduces to that of null controllability for the $1$D parabolic equations \eqref{eq::one-dim control system} (resp. 
 observability of the $1$D adjoint system \eqref{eq::one-dim parabolic equations}) that is uniform with respect to $n\in\N$.

The duality between null controllability and observability (see, for instance, \cite{dolecki1977} and \cite[Theorem 2.44]{coron2007}) tells us that if adjoint system \eqref{eq::homogeneous spherical parabolic BG equation} is observable from $\omega:=\omega_{x}\times[0,2\pi)$ (where $\omega_{x}\subset(-\pi/2,\pi/2)$) in time $T>0$ then the linear map
\begin{equation}\label{eq::control operator}
    \cU_T:f_0\in\operatorname{H}_\sigma\longmapsto\cU_T(f_0) = u\in L^2(0,T;\operatorname{H}_\sigma),
\end{equation}
where $u$ is the control supported in $(0, T)\times\omega$ that steers the solution of control system \eqref{eq::spherical parabolic BG equation} from the initial state $f_0$ to $0$ in time $T$ is well-defined and bounded.  

As demonstrated in \cite[Section 2.3]{tamekue2022null}, adjoint system \eqref{eq::homogeneous spherical parabolic BG equation} is observable in $\omega$ in time $T>0$ if and only if the family of $1$D adjoint system \eqref{eq::one-dim parabolic equations} is observable in $\omega_x$ in time $T>0$ uniformly with respect to $n\in\N$. This means that there exists a positive constant $C>0$ independent of $n\in\N$ such that the following uniform observability holds for system \eqref{eq::one-dim parabolic equations}
\begin{equation}\label{eq::uniform observability inequality}
	\int_{-\frac\pi 2}^{\frac\pi 2}|g_n(T,x)|^2dx\le C\int_{0}^{T}\int_{\omega_x}|g_n(t,x)|^{2}\cos xdxdt.
\end{equation}

Since the observability of $1$D adjoint system \eqref{eq::one-dim parabolic equations} is equivalent to null controllability of $1$D control system \eqref{eq::one-dim control system}, it follows that
\begin{equation}\label{eq::control operator diagonalisation}
	\cU_T = \bigoplus_{n\in\Z}^{\perp}\cU_{n,T}.
\end{equation}
Here the linear map $\cU_{n,T}$ is defined by
\begin{equation}\label{eq::control operator in 1D}
    \cU_{n,T}:f_{0,n}\in\cH_n\longmapsto\cU_{n,T}(f_{0,n}) = u_n\in L^2(0,T;\cH_n),
\end{equation}
where $u_n$ are Fourier components of the control $u$ defined in \eqref{eq::control operator} satisfying the following: there exists a positive constant $K>0$ independent of $n\in\N$ such that
\begin{equation}\label{eq::uniform null controllability estimates}
    \|u_n\|_{L^2(0,T;\cH_n)}\le K\|f_{0,n}\|_{\cH_n},\qquad\forall n\in\N.
\end{equation}
\begin{definition}[Uniform null controllability]
    Let $\omega_x\subset(-\pi/2,\pi/2)$. Control system \eqref{eq::one-dim control system} is null controllable from $\omega_x$ in time $T>0$ uniformly with respect to $n\in\N$ if there exists a control $u_n$ supported in $(0, T)\times\omega_x$ that steers the solution $f_n$ of \eqref{eq::one-dim control system} from $f_{0,n}$ to $0$ in time $T$ and satisfies \eqref{eq::uniform null controllability estimates}.
\end{definition}

Therefore, studying the null controllability (resp. observability) of control system \eqref{eq::spherical parabolic BG equation} (resp. adjoint system \eqref{eq::homogeneous spherical parabolic BG equation}) is equivalent to studying the uniform null controllability (resp. uniform observability) of a family of $1$D control systems \eqref{eq::one-dim control system} (resp. \eqref{eq::one-dim parabolic equations}).

\section{Observability when the degeneracy is inside the control set}\label{s:equator inside the control set}

This section proves Theorem~\ref{thm::main-spherical coordinates observability non-degenerate} for the control set $\Gamma=\Gamma_b\times[0,2\pi)$, where $\Gamma_b=(-b,b)$ and therefore $0\in\operatorname{int}(\Gamma_b)$. The symmetry of $\Gamma_b$ is imposed only for notational simplicity. By the reduction described in Section~\ref{ss::UNC}, it suffices to prove a uniform observability inequality for the one-dimensional adjoint family \eqref{eq::one-dim parabolic equations}. The precise statement is the following.

\begin{proposition}\label{pro::degenerate 1D observability inequality}
    Let $T>0$ and $\Gamma_b=(-b, b)$ be defined as in \eqref{eq::degenerate control set in spherical coordinates}. Then, there exists a positive constant $C>0$ such that for every $n\in\N$, the solution $g_n$ of~\eqref{eq::one-dim parabolic equations} satisfies
	\begin{equation}\label{eq::inequality n}
		\int_{-\frac\pi 2}^{\frac\pi 2}|g_n(T,x)|^2\cos xdx\le C\int_{0}^{T}\int_{\Gamma_b}|g_n(t,x)|^2\cos xdxdt.
	\end{equation}
\end{proposition}
\begin{remark}\label{rmk::key}
    It is already known from \cite[Proposition 3.4]{tamekue2022null} that the adjoint system \eqref{eq::one-dim parabolic equations} is observable from any subset $\Gamma\subset(-\pi/2,\pi/2)$ in an arbitrarily short time $T>0$ for the zero frequency. Therefore, it remains to prove Proposition~\ref{pro::degenerate 1D observability inequality} only for non-zero frequencies $n\in\N^{*}$.
\end{remark} 

For nonzero frequencies $n\in\N^*$, Proposition~\ref{pro::degenerate 1D observability inequality} follows from Carleman estimates whose constants are uniform with respect to $n$. As in~\cite[Section~3.2]{tamekue2022null}, we first remove the weight $\cos x\,dx$ by transferring the equation to the unweighted space $L^2(-\pi/2,\pi/2)$. Consider the unitary transformation 
\begin{eqnarray}\label{eq::unitary transformation U}
	\operatorname{U}:& L^2((-\pi/2,\pi/2);\cos xdx)&\longrightarrow  L^2(-\pi/2,\pi/2)\nonumber\\
	&v&\longmapsto (\operatorname{U}v)(x) = \sqrt{\cos x}v(x)
\end{eqnarray} 
and define for all $n\in\N^{*}$ the unbounded operator $\operatorname{M_n}$ on the space $ L^2(-\pi/2,\pi/2)$ by
\begin{equation}\label{eq::operators M_n}
	\operatorname{M_n} = \operatorname{U}\cL_{n}\operatorname{U}^{+},\qquad\qquad D(\operatorname{M_n}) = \operatorname{U}(D(\cL_{n})),
\end{equation} 
where $\operatorname{U}^{+}$ is the adjoint of the unitary operator $\operatorname{U}$. It follows that 
\begin{equation}\label{eq36}
	\operatorname{M_n}w = w''-q_{n}(x)w,\hspace{1cm}\forall w\in D(\operatorname{M_n}),
\end{equation} 
where, for all $n\in\N^{*}$, the potential $q_n$ is given by $q_n(x) = (n^2-1/4)\tan^2x-1/2$.

Since the differential operator $\partial_t$ commutes with the unitary transformation $\operatorname{U}$, one deduces that system \eqref{eq::one-dim parabolic equations} is equivalent to the following 
\begin{equation}\label{eq::transformer one-dim parabolic equation for n neq 0}
	\begin{cases}
		\partial_t\tilde{g}_n-\operatorname{M_n}\tilde{g}_n = 0,&\mbox{ in } (0,T)\times(-\pi/2,\pi/2),\cr
		\tilde{g}_n|_{t=0}  = \tilde{g}_{0,n},&\mbox{ in }(-\pi/2,\pi/2).
	\end{cases}
\end{equation}

\begin{remark}
	By standard parabolic smoothing for the homogeneous equation, the solution $\widetilde g_n$  of~\eqref{eq::transformer one-dim parabolic equation for n neq 0}
	belongs to
	$C([0,T];L^2(-\pi/2,\pi/2))\cap C^2((0,T);D(\operatorname{M}_n))$. In particular, $D(\operatorname{M_n})$ is the subset of the Sobolev space $H_0^1(-\pi/2,\pi/2)$ as shows the following result proved in \cite[Lemma 3.7]{tamekue2022null}.
\end{remark}

\begin{lemma}[\cite{tamekue2022null}]\label{lem::proprieties in D(M_n)}
	Let $n\in\N^{*}$ and $w\in D(\operatorname{M_n})$. Then $w'\in L^2(-\pi/2,\pi/2) $ and $w$ is locally absolutely continuous on $\displaystyle[-\pi/2,\pi/2]$. Moreover,
	\begin{equation}\label{eq::proprieties in D(M_n)}
	w(x)=o(1)\qquad\mbox{and}\qquad w'(x) = o(1)\qquad\mbox{both as}\quad x\to\pm\frac\pi 2.
	\end{equation}
\end{lemma}

We can now state the uniform Carleman estimate for \eqref{eq::transformer one-dim parabolic equation for n neq 0}. Combined with the dissipation estimate \eqref{eq::dissipation rate}, it yields Proposition~\ref{pro::degenerate 1D observability inequality} for all nonzero modes. The proof of the Carleman estimate is deferred to Section~\ref{s::UCE}. To simplify notation, we drop the tilde and the index $n$ and introduce, for each $n\in\N^*$, the singular parabolic operator
\begin{equation}\label{eq::singular 1D parabolic operator}
	\cP_n:=\partial_t- \partial_x^2+q_n(x)\qquad\mbox{with}\qquad q_n(x) = \left(n^2-\frac 1 4\right)\tan^2x-1/2.
\end{equation}
\begin{proposition}[Uniform Carleman estimate]\label{pro::uniform Carleman estimate}
	Let $\Gamma_b=(-b, b)$ be defined as in \eqref{eq::degenerate control set in spherical coordinates}. Then there exist a weight function $\beta\in C^{4}([-\pi/2,\pi/2])$ and positive constants $C_0, C_1>0$ such that for every $T>0$, $s\ge C_0(T+T^2)$ and $n\in\N^{*}$, every $\displaystyle g\in C([0,T]; L^2(-\pi/2,\pi/2))\cap C^2((0,T);D(\operatorname{M_{n}}))$ satisfies
	\begin{multline}\label{eq::uniform Carleman estimate}
		C_1\int_{0}^{T}\int_{-\frac\pi 2}^{\frac\pi 2}\left(s\theta(t)| \partial_x g(t,x)|^2+s^3\theta(t)^3|g(t,x)|^2\right)e^{-2s\varphi(t,x)}dxdt \\
		\le\int_{0}^{T}\int_{\Gamma_b}s^3\theta(t)^3|g(t,x)|^2e^{-2s\varphi(t,x)}dxdt
		+\int_{0}^{T}\int_{-\frac\pi 2}^{\frac\pi 2}|\cP_ng(t,x)|^2e^{-2s\varphi(t,x)}dxdt,
	\end{multline}
	where, $C_i:=C_i(\beta,b)$, $i=0,1$ and $\cP_n$ is defined in \eqref{eq::singular 1D parabolic operator}. Here we let
 \begin{equation}
 \theta(t) = 1/t(T-t)\quad\mbox{and}\quad \varphi(t,x) = \theta(t)\beta(x),\qquad t\in(0,T),\; x\in[-\pi/2,\pi/2].
 \end{equation}
\end{proposition}
\begin{remark}
	Notice that the main difference between the Carleman estimate provided in \cite[Proposition 4.2]{tamekue2022null} when the control set does not touch the degeneracy, and that given in  Proposition~\ref{pro::uniform Carleman estimate} when it contains the degeneracy in its interior is that the weight parameter $s$ does not depend on $n$ in the second case. This independence allows us to obtain a uniform observability result from $\Gamma_{b}$ in an arbitrarily short time $T>0$.
\end{remark}

\begin{proof}[\textit{Proof} of Proposition~\ref{pro::degenerate 1D observability inequality}]
We obtain the uniform observability inequality (38) for every $T>0$ for system (25) when $n\in\N^\ast$. Let $\tilde{g}_n = \operatorname{U}g_n\in C([0,T]; L^2(-\pi/2,\pi/2))\cap C^2((0,T);D(\operatorname{M_{n}}))$ be the solution of system \eqref{eq::transformer one-dim parabolic equation for n neq 0}, where $g_n$ is the Fourier component \eqref{eq::fourier coefficient of g} and $\operatorname{U}$, the unitary transformation defined in \eqref{eq::unitary transformation U}. Let $Q:=(0,T)\times(-\pi/2,\pi/2)$ and $dQ:=dxdt$. Then by Carleman estimate \eqref{eq::uniform Carleman estimate}, one deduces
	\begin{eqnarray}\label{eq114}
		C_1\int_{Q}\theta(t)^3|\tilde{g}_n(t,x)|^2e^{-2s\varphi}dQ&\le&\int_{0}^{T}\int_{\Gamma_{b}}\theta(t)^3|\tilde{g}_n(t,x)|^2e^{-2s\varphi}dQ,
	\end{eqnarray}
	for all $s\ge C_0(T+T^{2})$, and for some constants $C_0,C_1>0$ independent of $n$, $T$ and $\tilde{g}_n$. 
	For $t\in(T/3,2T/3)$, we have due to dissipation rate \eqref{eq::dissipation rate},
	$$\frac{4}{T^2}\le\theta(t)\le\frac{9}{2T^2}\hspace{0.5cm}\mbox{and}\hspace{0.5cm}\int_{-\frac\pi 2}^{\frac\pi 2}|\tilde{g}_n(T,x)|^2dx\le e^{-\frac{2}{3}nT}\int_{-\frac\pi 2}^{\frac\pi 2}|\tilde{g}_n(t,x)|^2dx.$$
	Integrating over $(T/3,2T/3)$, we find, using \eqref{eq114}
 \begin{eqnarray}
     \int_{-\frac\pi 2}^{\frac\pi 2}|\tilde{g}_n(T,x)|^2dx&\le&
		\frac{1}{C_3}\frac{T^5}{64}\frac{18}{8s^3\beta_{*}^3}e^{-\frac{2}{3}nT}e^{\frac{9}{T^2}s\beta^{*}}\int_{0}^{T}\int_{\Gamma_{b}}|\tilde{g}_n(t,x)|^2dxdt\nonumber\\
  &\le&C_3T^2e^{C_4\left(1+\frac 1 T\right)}\int_{0}^{T}\int_{\Gamma_{b}}|\tilde{g}_n(t,x)|^2dxdt,
 \end{eqnarray}
 where $\beta_{*}:=\min\{\beta(x):x\in[-\pi/2,\pi/2]\}$, $\beta^{*}:=\max\{\beta(x):x\in[-\pi/2,\pi/2]\}$ and  $C_3:=C_3(b)>0$ and $C_4:=C_4(b)>0$.
\end{proof}

\section{Minimal time}\label{s::minimal time}
Let $\omega=\omega_{a,b}\times[0,2\pi)$, with $\omega_{a,b}=(a,b)$ and $0\le a<b\le\pi/2$. The purpose of this section is to prove the upper-bound part of Theorem~\ref{thm::main-spherical coordinates degenerate}, namely, null controllability from $\omega$ for every
\[
T>\ln(1/\cos a).
\]
The argument reduces to the uniform null controllability of the one-dimensional systems \eqref{eq::one-dim control system} and uses a moment method. By Remark~\ref{rmk::key}, only the nonzero Fourier modes $n\in\N^*$ require a new argument. 

\subsection{Lower bound estimates involving the associated Legendre functions of the first kind}\label{ss::useful estimates}
We provide in this section useful lower-bound estimates on the $L^2(\omega_{a,\frac\pi 2},\cos xdx)$-norm of the associated Legendre functions of the first kind. Let $n\in\N^{*}$ and $\cL_n$ be defined by \eqref{eq::domain operator L_n}-\eqref{eq::operator L_n}. Recall that eigenfunctions of the operator $-\cL_n$ associated with eigenvalues $\lambda_{\ell,n} = \ell(\ell+1)-n^2$ are given by
$$
	v_{\ell,n}(x) = \sqrt{\frac{2\ell+1}{2}\frac{(\ell-n)!}{(\ell+n)!}}P_\ell^n(\sin x),\hspace{1cm}\forall x\in [-\pi/2,\pi/2]
$$
where $P_{\ell}^{n}$ are the associated Legendre functions of the first kind.
One has the following \(L^{2}\)-mass estimates.

\begin{proposition}\label{pro:weighted estimates}
Let $0\le a\le \tfrac{\pi}{2}$, fix $n\in\N^\ast$, and let $\ell\in\N^\ast$ with $\ell\ge n$. Then, it holds
\begin{equation}\label{eq:L2 mass estimate}
    \int_a^{\frac{\pi}{2}}|v_{\ell,n}(x)|^2\,\cos{x}\,dx\ge C_{\ell,n}\cos^{2n+2}{a}
\end{equation}
where
\begin{equation}\label{eq:Cln}
    C_{\ell,n}:=\frac{(2\ell+1)(n+1)}{2^{2n+2}}\frac{(\ell-n)!(\ell+n)!}{((\ell+1)!)^2}.
\end{equation}
\end{proposition}

\begin{proof}
The inequality is obviously satisfied for $a=\pi/2$. Assume now that $0\le a<\pi/2$. We start from the change of variables $t=\sin{x}$, $dt=\cos{x}\,dx$. Then
\[
\int_a^{\frac{\pi}{2}}|v_{\ell,n}(x)|^2\,\cos{x}\,dx
= \frac{2\ell+1}{2}\frac{(\ell-n)!}{(\ell+n)!}
\int_{\sin a}^1\big(P_\ell^n(t)\big)^2\,dt.
\]
Let $P_\ell$ denote the Legendre polynomial of degree $\ell$. Using the standard formula \cite[p.~327]{courant1953}
\[
P_\ell^n(t)=(-1)^n(1-t^2)^{n/2}Q_{\ell,n}(t),
\qquad Q_{\ell,n}(t):=P_\ell^{(n)}(t)=\frac{d^n}{dt^n}P_\ell(t),
\]
we obtain
\begin{equation}\label{eq:inter}
\int_a^{\frac{\pi}{2}}|v_{\ell,n}(x)|^2\,\cos{x}\,dx
= \frac{2\ell+1}{2}\frac{(\ell-n)!}{(\ell+n)!}
\int_{\sin a}^1(1-t^2)^nQ_{\ell,n}(t)^2\,dt.
\end{equation}
Set $x=1-t$ and $y=1-\sin a$. Then $x\in[0,y]$ and $1-t^2=x(2-x)$. With
\[
R(x):=Q_{\ell,n}(1-x)=P_\ell^{(n)}(1-x),
\]
we have
\begin{equation}\label{eq:inter1}
\int_{\sin a}^1(1-t^2)^nQ_{\ell,n}(t)^2\,dt
=\int_0^y x^n(2-x)^nR(x)^2\,dx
=y^{n+1}\int_0^1u^n(2-yu)^nR(yu)^2\,du.
\end{equation}
Since $0\le u\le1$, one has $2-yu\ge 2-y=1+\sin a$. Hence
\[
\int_0^yx^n(2-x)^nR(x)^2\,dx
\ge y^{n+1}(1+\sin a)^n\int_0^1u^nR(yu)^2\,du.
\]
Moreover,
\[
y^{n+1}(1+\sin a)^n
=\frac{\cos^{2n+2}a}{1+\sin a}
\ge \frac12\cos^{2n+2}a.
\]
Thus
\begin{equation}\label{eq:inter2}
\int_0^yx^n(2-x)^nR(x)^2\,dx
\ge\frac{\cos^{2n+2}a}{2}\int_0^1u^nR(yu)^2\,du.
\end{equation}

It remains to estimate the last integral from below uniformly with respect to $y\in[0,1]$. Define
\[
S(u):=R(yu),\qquad u\in[0,1].
\]
Then $S$ is a polynomial of degree at most $\ell-n$, and
\[
S(0)=R(0)=P_\ell^{(n)}(1)
=\frac{(\ell+n)!}{2^n n!(\ell-n)!}=:c_0.
\]
We use the Christoffel-function estimate for the space of polynomials of degree at most $\ell-n$ with respect to the weight $u^n$ on $[0,1]$. Let $(p_k)_{k=0}^{\ell-n}$ be an orthonormal basis of this weighted polynomial space. Expanding
\[
S(u)=\sum_{k=0}^{\ell-n}a_kp_k(u),
\]
Cauchy-Schwarz's inequality gives
\[
|S(0)|^2
\le \left(\int_0^1u^nS(u)^2\,du\right)
\left(\sum_{k=0}^{\ell-n}p_k(0)^2\right).
\]
Consequently,
\[
\int_0^1u^nS(u)^2\,du
\ge \frac{c_0^2}{K_{\ell,n}(0,0)},
\qquad
K_{\ell,n}(0,0):=\sum_{k=0}^{\ell-n}p_k(0)^2.
\]
For the weight $u^n$ on $[0,1]$, we choose the shifted Jacobi polynomials
\begin{equation}\label{eq:shifted Jacobi polynomials}
p_k(u)=\sqrt{2k+n+1}\,P_k^{(0,n)}(2u-1).
\end{equation}
Then
\[
p_k(0)^2=(2k+n+1)\binom{k+n}{k}^2,
\]
because $P_k^{(0,n)}(-1)=(-1)^k\binom{k+n}{k}$. Hence
\[
K_{\ell,n}(0,0)
=\sum_{k=0}^{\ell-n}(2k+n+1)\binom{k+n}{k}^2.
\]
Using the identity
\[
(2k+n+1)\binom{k+n}{k}^2
=(n+1)\left[
\binom{k+n+1}{n+1}^2-
\binom{k+n}{n+1}^2
\right],
\]
we obtain, by telescoping,
\[
K_{\ell,n}(0,0)=(n+1)\binom{\ell+1}{n+1}^2.
\]
Therefore
\begin{equation}\label{eq:christoffel-lower}
\int_0^1u^nR(yu)^2\,du
\ge \frac{c_0^2}{(n+1)\binom{\ell+1}{n+1}^2}.
\end{equation}
Combining \eqref{eq:inter}, \eqref{eq:inter1}, \eqref{eq:inter2}, and \eqref{eq:christoffel-lower} yields
\[
\int_a^{\frac{\pi}{2}}|v_{\ell,n}(x)|^2\,\cos{x}\,dx
\ge
\frac{2\ell+1}{2}\frac{(\ell-n)!}{(\ell+n)!}
\frac{c_0^2}{2(n+1)\binom{\ell+1}{n+1}^2}
\cos^{2n+2}a.
\]
Substituting
\[
c_0=\frac{(\ell+n)!}{2^n n!(\ell-n)!},
\qquad
\binom{\ell+1}{n+1}=\frac{(\ell+1)!}{(n+1)!(\ell-n)!},
\]
gives exactly \eqref{eq:Cln}. This completes the proof.
\end{proof}

\subsection{Moment method}\label{ss::MM}

This section is devoted to the moment method for the boundary-touching crown and proves the first part of Theorem~2.6 in the case $\omega=(a,\pi/2)\times[0,2\pi)$.

Recall from Section~\ref{ss::FE} that if $f_0\in\operatorname{H_{\sigma}}$ is the initial state in system~\eqref{eq::spherical parabolic BG equation}, we denote by $f_{0,n}\in\cH_{n}\cong  L^2((-\pi/2,\pi/2);\cos xdx)$ its Fourier coefficients. The eigenfunctions $(v_{\ell,n}(x))_{\ell\ge n\ge 1}$ of operator $\cL_n$ associated with eigenvalues $\lambda_{\ell,n}$ form a complete orthonormal set of the Hilbert space $\cH_{n}$. 
We expand 
\begin{equation}\label{eq::fourier coef of f_n in H_n}
    f_{0,n} = \sum\limits_{\ell\ge n} f_{0,n}^\ell v_{\ell,n},\qquad\qquad f_{0,n}^\ell :=\langle f_{0,n}, v_{\ell,n}\rangle_{\cH_n},\qquad\forall n\in\N^{*}.
\end{equation}

The solution $f_n$ of $1$D control system \eqref{eq::one-dim control system} with $\omega_x=\omega_{a,b}$ is given in time $T>0$ by
\begin{equation}\label{eq::one-dim control system solution}
	f_n(T,x) = e^{T\cL_n}f_{0,n}(x)+\int_{0}^{T}e^{(T-t)\cL_n}u_n(t,x)\mathbf 1_{\omega_{a,b}}(x)dt,\qquad x\in (-\pi/2,\pi/2).
\end{equation}
Multiplying~\eqref{eq::one-dim control system solution} by $v_{\ell,n}$ and integrating over $(-\pi/2,\pi/2)$ (with respect to the measure $\cos xdx$) leads to
\begin{equation}
    \langle f_n(T),v_{\ell,n}\rangle_{\cH_n}=e^{-\lambda_{\ell,n}T}f_{0,n}^\ell+\int_{0}^{T}\int_{\omega_{a,b}}e^{-\lambda_{\ell,n}(T-t)}u_n(t,x)v_{\ell,n}(x)\cos xdxdt, \quad\forall\ell\ge n\ge 1.
\end{equation}
It follows that finding a control $u_n\in L^2(0,T;\cH_n)$ such that $f_n(T)=0$
amounts to finding $u_n\in L^2(0,T;\cH_n)$ satisfying the following moment
conditions
\begin{equation}\label{eq::moment problem}
    -e^{-\lambda_{\ell,n}T}f_{0,n}^\ell
    = \int_{0}^{T}\int_{\omega_{a,b}}
    e^{-\lambda_{\ell,n}(T-t)}u_n(t,x)v_{\ell,n}(x)\cos x\,dx\,dt,
    \qquad \forall \ell\ge n\ge 1.
\end{equation}

Let $(q_\ell^n(t))_{\ell\ge n}$ be a family of functions biorthogonal to the family of functions $(e^{-\lambda_{\ell,n}(T-t)})_{\ell\ge n}$ in the Hilbert space $L^2(0, T)$. One has
\begin{equation}\label{eq::orthogonality relation}
    \int_{0}^{T}q_k^n(t)e^{-\lambda_{\ell,n}(T-t)}dt = \delta_{k,\ell},\qquad\forall k,\;\ell\ge n,
\end{equation}
where $\delta_{k,\ell}$ is the Kronecker symbol.
We look for a control $u_n$ of the form 
$$u_n(t,x)=\sum\limits_{\ell=n}^{+\infty}\alpha_{\ell,n}\,q_{\ell}^n(t)v_{\ell,n}(x)\mathbf 1_{\omega_{a,b}}(x),$$
for some sequence of real numbers $(\alpha_{\ell,n}\,)_{\ell\ge n}$. We then deduce from
\eqref{eq::moment problem} and \eqref{eq::orthogonality relation} that the control $u_n$ is formally given for every $t\in(0,T)$ and every $x\in(-\pi/2,\pi/2)$ by
\begin{equation}\label{eq::1D control functions}
    u_n(t,x) = -\sum\limits_{\ell=n}^{+\infty}\frac{\displaystyle e^{-\lambda_{\ell,n}T}\langle f_{0,n}, v_{\ell,n}\rangle_{\cH_n}}{\displaystyle\|v_{\ell,n}\mathbf 1_{\omega_{a,b}}\|_{\cH_n}^2}q_{\ell}^n(t)v_{\ell,n}(x)\mathbf1_{\omega_{a,b}}(x),\qquad\qquad\forall n\in\N^{*}.
\end{equation}

\begin{remark}\label{rmk::important uniform NC}
   It remains to rigorously justify computations leading to \eqref{eq::1D control functions}, which consists of
\begin{enumerate}
    \item showing the existence of a family of functions $(q_\ell^n(t))_{\ell\ge n}$ that are biorthogonal to the family of functions $(e^{-\lambda_{\ell,n}(T-t)})_{\ell\ge n}$;
    \item showing that $u_n$ defined by \eqref{eq::1D control functions} yields the uniform null controllability of $1$D control problem \eqref{eq::one-dim control system}. By construction, it remains to show that $u_n\in L^2(0,T;\cH_n)$ and that \eqref{eq::uniform null controllability estimates} holds. 
\end{enumerate}
\end{remark}

Recall from \cite[Theorem 1.1]{fattorini1974} the following result, which gives a sufficient condition for bounded from above (that is, uniform in $n$ and depends exponentially on $\ell$), the $L^2(0, T)$ norm of $q_\ell^n$.

\begin{theorem}[\cite{fattorini1974}]\label{thm::fattorini-Russel}
    Let $\rho>0$ and $\cN:\varepsilon\in(0,+\infty)\mapsto\cN(\varepsilon)\in\N^{*}$. Denote by $\cL(\rho,\cN)$ the class of all sequences of positive numbers $(\sigma_\ell)_{\ell\in\N^{*}}$ that satisfy the conditions
    \begin{align}
        \sigma_1\ge\rho,\qquad\qquad\sigma_{\ell+1}-\sigma_\ell\ge\rho, \label{eq::gap condition}\\
        \sum\limits_{\ell=\cN(\varepsilon)}^{+\infty}\frac{1}{\sigma_{\ell}}\le\epsilon,\qquad\varepsilon>0.\label{eq::finite sum of inverse}
    \end{align}
    Then there is a positive function $K(\varepsilon)$ defined for $\varepsilon>0$ and determined solely by $T>0$, $\rho$ and $\cN$, such that
    \begin{equation}\label{eq::upper bound q}
        \|q_\ell^n\|_{L^2(0,T)}\le K(\varepsilon)e^{\varepsilon\sigma_\ell},\qquad\forall\ell\in \N^{*}.
    \end{equation}
\end{theorem}
\begin{remark}
 Fix $n\in\N^{*}$ and let $(\sigma_\ell)_{\ell\ge n}$ be a sequence of real numbers. The family of functions $(q_\ell^n(t))_{\ell\ge n}$ that are biorthogonal to the family of functions $(e^{-\sigma_{\ell}(T-t)})_{\ell\ge n}$ exists only if the following condition from \cite{schwartz1959} is verified,
\begin{equation}\label{eq::Schwartz condition}
    \sum\limits_{\ell=n}^{+\infty}\frac{1}{\sigma_{\ell}}<\infty,\qquad\qquad\forall n\in\N^{*}.
\end{equation}
In the case at hand, it is immediate that \eqref{eq::Schwartz condition} holds for the sequence of eigenvalues $(\lambda_{\ell,n})_{\ell\ge n}$ since $\lambda_{\ell, n}\sim\ell^2$.
\end{remark}

\begin{lemma}\label{lem::uniform lambda FR}
    There exists $\rho>0$ and a map $\cN:\varepsilon\in(0,+\infty)\mapsto\cN(\varepsilon)\in\N^{*}$ such that for every $n\in\N^{*}$, the reindexed sequence 
    \begin{equation}
        \sigma_j^{(n)}:=\lambda_{n+j-1,n},\qquad j\ge 1
    \end{equation}
    belongs to the class $\cL(\rho,\cN)$ of Theorem~\ref{thm::fattorini-Russel}.
\end{lemma}
\begin{proof}
For every $n\ge1$, one has
\[
\sigma_1^{(n)}=\lambda_{n,n}=n\ge1.
\]
Moreover,
\[
\sigma_{j+1}^{(n)}-\sigma_j^{(n)}
=
\lambda_{n+j,n}-\lambda_{n+j-1,n}
=
2(n+j)\ge2.
\]
Thus, \eqref{eq::gap condition} holds uniformly with $\rho=1$. It remains to verify~\eqref{eq::finite sum of inverse} uniformly in $n$. Since
\[
\sigma_j^{(n)}
=
\lambda_{n+j-1,n}
=
(n+j-1)(n+j)-n^2,
\]
we have
\[
\sigma_j^{(n)}
=
n+(2n+1)(j-1)+(j-1)^2.
\]
In particular, for $j\ge2$,
\[
\sigma_j^{(n)}\ge (j-1)^2.
\]
Therefore, for $N\ge2$,
\[
\sum_{j=N}^{+\infty}\frac{1}{\sigma_j^{(n)}}
\le
\sum_{j=N}^{+\infty}\frac{1}{(j-1)^2},
\]
uniformly in $n$. Hence, for every $\eta>0$, we can choose $N(\eta)\ge2$ large enough so that
\[
\sum_{j=N(\eta)}^{+\infty}\frac{1}{(j-1)^2}\le \eta.
\]
It follows that
\[
\sum_{j=N(\eta)}^{+\infty}\frac{1}{\sigma_j^{(n)}}\le \eta,
\qquad \forall n\ge1.
\]
This proves the claim.
\end{proof}

We now prove point 2 of Remark~\ref{rmk::important uniform NC}.
\begin{proposition}\label{pro::proposition MM}
Let $\varepsilon>0$, $0\le a<\pi/2$ and $T>\varepsilon+\ln(1/\cos a)$. There exists a constant $C_0=C_0(a,T,\varepsilon)>0$ such that, for every $n\in\N^{*}$, the control $u_n$ defined in \eqref{eq::1D control functions} satisfies
\begin{equation}
    \|u_n\|_{L^2(0,T;\cH_n)}\ \le\ C_0\,\|f_{0,n}\|_{\cH_n}.
\end{equation}
\end{proposition}

\begin{proof}
Fix $n\in\N^{*}$ and set
\[
g_{\ell,n}:=v_{\ell,n}\mathbf 1_{\omega_{a,\pi/2}},\qquad c_{n,\ell}:=-\frac{\displaystyle e^{-\lambda_{\ell,n}T}\langle f_{0,n}, v_{\ell,n}\rangle_{\cH_n}}{\displaystyle\|g_{\ell,n}\|_{\cH_n}^2}.
\]
For $N\ge n$, define the partial sum
\[
u_n^N(t,\cdot):=\sum_{\ell=n}^{N}c_{n,\ell}q_{\ell}^{\,n}(t)g_{\ell,n}.
\]
By the triangle inequality in $L^2(0,T;\cH_n)$ and by the definition of $c_{n,\ell}$,
\[
\|u_n^N\|_{L^2(0,T;\cH_n)}
\le
\sum_{\ell=n}^{N}
e^{-T\lambda_{\ell,n}}
|\langle f_{0,n},v_{\ell,n}\rangle_{\cH_n}|
\frac{\|q_{\ell}^{\,n}\|_{L^2(0,T)}}{\|g_{\ell,n}\|_{\cH_n}}.
\]
By Lemma~\ref{lem::uniform lambda FR} and Theorem~\ref{thm::fattorini-Russel}, there exists a constant $K(\varepsilon)>0$, independent of $n$ and $\ell$, such that
\[
\|q_{\ell}^{\,n}\|_{L^2(0,T)}\le K(\varepsilon)e^{\varepsilon\lambda_{\ell,n}}.
\]
Thus, with $\beta:=T-\varepsilon$,
\[
\|u_n^N\|_{L^2(0,T;\cH_n)}
\le
K(\varepsilon)
\sum_{\ell=n}^{N}
e^{-\beta\lambda_{\ell,n}}
|\langle f_{0,n},v_{\ell,n}\rangle_{\cH_n}|
\frac{1}{\|g_{\ell,n}\|_{\cH_n}}.
\]
Cauchy-Schwarz's inequality and Bessel's inequality give
\begin{equation}\label{eq:MM-partial-estimate}
\|u_n^N\|_{L^2(0,T;\cH_n)}
\le
K(\varepsilon)\|f_{0,n}\|_{\cH_n}
\left(
\sum_{\ell=n}^{N}
\frac{e^{-2\beta\lambda_{\ell,n}}}{\|g_{\ell,n}\|_{\cH_n}^2}
\right)^{1/2}.
\end{equation}
It remains to prove that
\[
S_n:=\sum_{\ell=n}^{+\infty}
\frac{e^{-2\beta\lambda_{\ell,n}}}{\|g_{\ell,n}\|_{\cH_n}^2}
\]
is bounded uniformly in $n$. By the localized lower bound \eqref{eq:L2 mass estimate},
\[
S_n\le
\cos^{-2n-2}a
\sum_{\ell=n}^{+\infty}\frac{e^{-2\beta\lambda_{\ell,n}}}{C_{\ell,n}},
\qquad
C_{\ell,n}=\frac{(2\ell+1)(n+1)}{2^{2n+2}}
\frac{(\ell-n)!(\ell+n)!}{((\ell+1)!)^2}.
\]
Writing $\ell=n+m$, $m\ge0$, we have $\lambda_{n+m,n}=m(2n+1)+n+m^2$.
Moreover, Lemma~\ref{lem:A-over-power} gives
\[
\frac{1}{C_{n+m,n}}
\le
\frac{4(2n+1)(n+m+1)^2(n+m)^m}{(2n+2m+1)(n+1)}.
\]
Let $\alpha:=\ln(1/\cos a)$,
so that $\cos^{-2n-2}a=e^{2(n+1)\alpha}$. By assumption, $\beta> \alpha$. Hence
\begin{equation}\label{eq:Sn-pre-bound}
S_n\le
\frac{4(2n+1)}{n+1}e^{2(n+1)\alpha}
\sum_{m=0}^{+\infty}
\frac{(n+m+1)^2}{2n+2m+1}
e^{m\ln(n+m)}e^{-2\beta(m(2n+1)+n+m^2)}.
\end{equation}
Since
\[
\frac{2n+1}{n+1}\le2,
\qquad
\frac{(n+m+1)^2}{2n+2m+1}\le n+m+1,
\]
we obtain
\begin{equation}\label{eq:Sn-simple-bound}
S_n\le
8e^{2\alpha}
\sum_{m=0}^{+\infty}(n+m+1)
\exp\Big(-2(\beta-\alpha)n-2\beta m(2n+1)-2\beta m^2+m\ln(n+m)\Big).
\end{equation}
The term corresponding to $m=0$ is uniformly bounded in $n$. For $m\ge1$, we use the elementary estimate
\[
\ln r\le \beta r+\ln(1/\beta)-1,
\qquad r>0.
\]
Thus
\[
m\ln(n+m)\le \beta mn+\beta m^2+\bigl(\ln(1/\beta)-1\bigr)m.
\]
Therefore, for $m\ge1$,
\begin{equation*}
    -2(\beta-\alpha)n-2\beta m(2n+1)-2\beta m^2+m\ln(n+m)\le -2(\beta-\alpha)n-3\beta mn-\beta m^2+\bigl(\ln(1/\beta)-1-2\beta\bigr)m.
\end{equation*}
Using $n+m+1\le(n+1)(m+1)$ and dropping the negative term $-3\beta mn$, we get
\[
\sum_{m=1}^{+\infty}(n+m+1)e^{-2(\beta-\alpha)n-2\beta m(2n+1)-2\beta m^2+m\ln(n+m)}
\le
(n+1)e^{-2(\beta-\alpha)n}\sum_{m=1}^{+\infty}(m+1)e^{-\beta m^2+\gamma_\beta m},
\]
where $\gamma_\beta:=\ln(1/\beta)-1-2\beta$. The sequence $(n+1)e^{-2(\beta-\alpha)n}$ is bounded because $\beta>\alpha$, and the series $\sum_{m\ge1}(m+1)e^{-\beta m^2+\gamma_\beta m}$ converges. Consequently,
\[
M_{a,T,\varepsilon}:=\sup_{n\ge1}S_n<+\infty.
\]
It follows from \eqref{eq:MM-partial-estimate} that
\[
\|u_n^N\|_{L^2(0,T;\cH_n)}
\le
K(\varepsilon)M_{a,T,\varepsilon}^{1/2}\|f_{0,n}\|_{\cH_n},
\qquad N\ge n.
\]
Applying the same estimate to $u_n^{N_2}-u_n^{N_1}$ and using Bessel's inequality on the tail of the Fourier coefficients shows that $(u_n^N)_{N\ge n}$ is Cauchy in $L^2(0,T;\cH_n)$. Its limit is precisely the series defining $u_n$, and passing to the limit gives
\[
\|u_n\|_{L^2(0,T;\cH_n)}
\le
C_0\|f_{0,n}\|_{\cH_n},
\qquad
C_0:=K(\varepsilon)M_{a,T,\varepsilon}^{1/2}.
\]
The constant $C_0$ depends only on $a,T,\varepsilon$, and the proof is complete.
\end{proof}

\begin{remark}
Proposition~\ref{pro::proposition MM} therefore yields the first part of Theorem~\ref{thm::main-spherical coordinates degenerate} for the boundary-touching strip $\omega=(a,\pi/2)\times[0,2\pi)$.
\end{remark}

\subsection{Cut-off argument}\label{ss:cut-off argument}
In this section, we remove the restriction $b=\pi/2$ by a cut-off argument on the full domain to complete the proof of the first part of Theorem~\ref{thm::main-spherical coordinates degenerate}.

\begin{proposition}\label{pro::cutoff-reduction}
Let $0\le a<b<\pi/2$. Then system~\eqref{eq::spherical parabolic BG equation}
is null controllable from $\omega_{a,b}\times[0,2\pi)$ in time $T>\ln(1/\cos a)$.
\end{proposition}

\begin{proof}
Fix $0\le a<b<\pi/2$ and $T>\ln(1/\cos a)$. Choose two numbers $a<\alpha<\beta<b$,
and let $\chi\in C^\infty([-\pi/2,\pi/2])$ satisfy
\[
\chi(x)=0 \quad\text{for }x\le \alpha,
\qquad
\chi(x)=1 \quad\text{for }x\ge \beta .
\]
Then
\[
\operatorname{supp}\chi' \subset (\alpha,\beta)\subset(a,b).
\]
Let $f_0\in L^2(\Omega;\sigma)$. Since $T>\ln(1/\cos a)$,
Proposition~\ref{pro::proposition MM} yields the null controllability of
\eqref{eq::spherical parabolic BG equation} from
\[
\omega_R:=(a,\pi/2)\times[0,2\pi).
\]
Hence there exists a control
\[
u_R\in L^2(0,T;L^2(\Omega;\sigma)),
\qquad
\operatorname{supp} u_R \subset (0,T)\times \omega_R,
\]
such that the corresponding solution $f^R$ of
\[
\begin{cases}
\partial_t f^R-\Delta_{\operatorname{BG}}f^R=u_R\mathbf 1_{\omega_R},
&\text{in }(0,T)\times\Omega,\\
f^R|_{t=0}=f_0,
\end{cases}
\]
satisfies $f^R(T,\cdot,\cdot)=0$. By Theorem~\ref{thm::intrinsic SBG non-degenerate}, system
\eqref{eq::spherical parabolic BG equation} is null controllable in every
positive time from the symmetric strip
\[
\Gamma:=(-b,b)\times[0,2\pi).
\]
In particular, it is null controllable from $\Gamma$ in time $T>\ln(1/\cos a)$. Therefore,
there exists a control
\[
u_L\in L^2(0,T;L^2(\Omega;\sigma)),
\qquad
\operatorname{supp} u_L \subset (0,T)\times \Gamma,
\]
such that the corresponding solution $f^L$ of
\[
\begin{cases}
\partial_t f^L-\Delta_{\operatorname{BG}}f^L=u_L\mathbf 1_{\Gamma},
&\text{in }(0,T)\times\Omega,\\
f^L|_{t=0}=f_0,
\end{cases}
\]
satisfies $f^L(T,\cdot,\cdot)=0$. Define
\[
f:=(1-\chi)f^R+\chi f^L.
\]
Since $\chi$ depends only on $x$, one has
\[
f|_{t=0}=f_0
\qquad\text{and}\qquad
f(T,\cdot,\cdot)=0.
\]
Using $\Delta_{\operatorname{BG}}
=
\partial_x^2-\tan x\,\partial_x+\tan^2x\,\partial_y^2$,
one checks that, for every smooth function $z$,
\[
[\Delta_{\operatorname{BG}},\chi]z
:=
\Delta_{\operatorname{BG}}(\chi z)-\chi \Delta_{\operatorname{BG}}z
=
2\chi'(x)\partial_x z+\bigl(\chi''(x)-\tan x\,\chi'(x)\bigr)z .
\]
Therefore, in the sense of distributions on $(0,T)\times\Omega$, $\partial_t f-\Delta_{\operatorname{BG}}f=u$,
where
\begin{align*}
u
&:=(1-\chi)\,u_R\mathbf 1_{\omega_R}
   +\chi\,u_L\mathbf 1_{\Gamma}
   +[\Delta_{\operatorname{BG}},\chi](f^R-f^L) \\
&=(1-\chi)\,u_R\mathbf 1_{\omega_R}
   +\chi\,u_L\mathbf 1_{\Gamma}
   +2\chi'(x)\partial_x(f^R-f^L)
   +\bigl(\chi''(x)-\tan x\,\chi'(x)\bigr)(f^R-f^L).
\end{align*}
It remains to check that $u$ belongs to $L^2(0,T;L^2(\Omega;\sigma))$ and is
supported in $(0,T)\times\omega_{a,b}\times[0,2\pi)$. The support property follows
from the following three facts:
\begin{itemize}
\item[(i)] $(1-\chi)\,u_R\mathbf 1_{\omega_R}$ is supported in
$(a,\beta)\times[0,2\pi)\subset \omega_{a,b}\times[0,2\pi)$, since
$\chi\equiv1$ on $[\beta,\pi/2]$;
\item[(ii)] $\chi\,u_L\mathbf 1_{\Gamma}$ is supported in
$(\alpha,b)\times[0,2\pi)\subset \omega_{a,b}\times[0,2\pi)$, since
$\chi\equiv0$ on $[-\pi/2,\alpha]$ and $\Gamma=(-b,b)\times[0,2\pi)$;
\item[(iii)] the commutator term is supported in
$\operatorname{supp}\chi' \subset (\alpha,\beta)\subset(a,b)$.
\end{itemize}
To justify the $L^2$ regularity of the commutator term, we use
Proposition~\ref{pro::well-posedness}, which yields $f^R,f^L\in C([0,T];\operatorname H_\sigma)\cap L^2(0,T;\mathcal V_\sigma)$. By density of smooth functions in $V_\sigma$, identity~\eqref{eq::quadratic-form-BG} extends to every element of $V_\sigma$.
Since $\operatorname{supp}\chi'\cup\operatorname{supp}\chi''\subset(\alpha,\beta)\subset\subset(-\pi/2,\pi/2)$, the
functions $\chi'$, $\chi''$, and $\tan x\,\chi'$ are bounded on their support.
Moreover, by identity~\eqref{eq::quadratic-form-BG}, valid for all elements of $V_\sigma$, the quantity $\partial_x(f^R-f^L)$ belongs to $L^2\bigl((0,T)\times(\alpha,\beta)\times(0,2\pi);\sigma\bigr)$. Hence
\[
2\chi'(x)\partial_x(f^R-f^L)\in L^2(0,T;L^2(\Omega;\sigma))\quad\text{and}\quad \bigl(\chi''(x)-\tan x\,\chi'(x)\bigr)(f^R-f^L)\in L^2(0,T;L^2(\Omega;\sigma)),
\]
because $f^R-f^L\in C([0,T];\operatorname H_\sigma)\subset
L^2(0,T;\operatorname H_\sigma)$. Therefore $u\in L^2(0,T;L^2(\Omega;\sigma))$. Thus $u$ is an admissible control supported in
$(0,T)\times\omega_{a,b}\times[0,2\pi)$ which steers the solution of
\eqref{eq::spherical parabolic BG equation} from $f_0$ to $0$ in time $T$.
\end{proof}

\appendix

\section{Proof of the uniform Carleman estimates when the degeneracy is inside the control set}\label{s::UCE}
This section is dedicated to proving Proposition~\ref{pro::uniform Carleman estimate}. We follow the same line as in the proof of \cite[Proposition 4.1]{tamekue2022null}. For clarity, we introduce the following notations.
\begin{notation}
	We let $b'$ be a real number such that
	\begin{equation}\label{eq::parameters}
		0<b'<b\le\pi/2\qquad\mbox{and}\qquad [-b',b']\subset\Gamma_b:=(-b,b).
	\end{equation}
	We consider the subdomains
	\begin{equation}\label{eq::subdomains}
		\operatorname{\omega}_{con}:=(-b',b'),\qquad\qquad\operatorname{\omega}_{bdy}:=\left(-\pi/2,-b'\right)\cup\left(b',\pi/2\right),
	\end{equation}
	so that $$(-\pi/2,\pi/2) = \operatorname{\omega}_{bdy}\cup\operatorname{\omega}_{con}\qquad\qquad\mbox{and}\qquad\qquad \operatorname{\omega}_{con}\subset\subset\Gamma_b.$$
		We also introduce the weight function
	\begin{equation}\label{eq51}
		\varphi(t,x) = s\theta(t)\beta(x),\qquad\qquad(t,x)\in Q:=(0,T)\times I,\qquad I:=(-\pi/2,\pi/2),
	\end{equation}
	where the positive constant $s = s(T,\beta)>0$ will be chosen later on and the temporal weight $\theta$ is given by
	\begin{equation}\label{eq52}
		\theta(t) = \frac{1}{t(T-t)},\qquad\qquad t\in(0,T).
	\end{equation}
	Finally, we introduce for all $n\in\N^{*}$ and every $g\in C([0,T]; L^2(-\pi/2,\pi/2))\cap C^{2}((0,T);D(\operatorname{M_{n}}))$,  the change of function
	\begin{equation}\label{eq53}
		z(t,x) = g(t,x)e^{-\varphi(t,x)},\qquad (t,x)\in Q.
	\end{equation}
\end{notation}
We begin with the design of the weight function $\beta$ as depicted in Figure~\ref{fig:: uniform Carleman weight beta}.

\begin{lemma}\label{lem::uniform carleman weight beta}
	The function $\beta\in C^{4}([-\pi/2,\pi/2])$ satisfies
    \begin{align}
    \beta(x) = \ln|\sin x|+2|x|+2\qquad\qquad x\in[-\pi/2,-b']\cup[b',\pi/2],\label{eq::on -pi/2, -b'}\\
    \beta'(x)\operatorname{sign}(x)\ge 0\qquad\qquad\qquad\qquad x\in[-b',-b'/2]\cup[b'/2,b'],\\
    \beta(x) = 1   \qquad\qquad\qquad\qquad\qquad\qquad\qquad  x\in[-b'/2, b'/2].
\end{align}
	
	Consequently, it holds
    \begin{align}
        \beta\ge 1,\qquad\qquad\qquad\mbox{on}\qquad\qquad\qquad(-\pi/2,\pi/2),\label{eq::beta on -pi/2, pi/2}\\
        \beta'(x)\sin x\ge 0  \qquad\qquad\qquad\qquad\qquad\qquad  x\in[-b',b'].\label{eq::key properties of beta in omega_con}
    \end{align}

\end{lemma}
\begin{figure}
	\centering
	\includegraphics[width=0.45\textwidth]{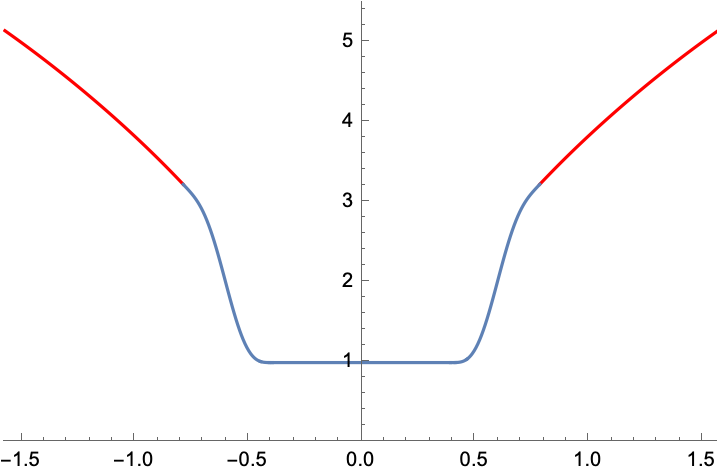}
	\caption{ The weight $\beta$ depicted on $[-\pi/2,\pi/2]$. The curves parts in \textit{blue} and \textit{red} correspond respectively to the subcontrol region $\operatorname{\omega}_{con}$ containing the degeneracy in its interior and the boundary domain $\operatorname{\omega}_{bdy}$ containing the singular points $\pm\pi/2$.}
	\label{fig:: uniform Carleman weight beta}
\end{figure}
\begin{remark}
	The weight $\beta$ in Lemma~\ref{lem::uniform carleman weight beta} cannot be used if the control set contains the degeneracy on its boundary, say, $\Gamma_b:= (0, b)$ with $0<b\le\pi/2$. Indeed, the expression of $\beta$ given by  \eqref{eq::on -pi/2, -b'} explodes at $0$.
\end{remark}

We let 
	\begin{equation}\label{eq59}
		\cP_n^{+}z+\cP_n^{-}z = e^{-\varphi}\cP_ng,
	\end{equation}
	where $\cP_n$ is the parabolic operator introduced in \eqref{eq::singular 1D parabolic operator}, and
	\begin{equation}\label{eq60}
		\cP_n^{+}z = -\operatorname{M_n}z+(\partial_{t}\varphi-|\partial_x\varphi|^2)z\qquad\mbox{and}\qquad\cP_n^{-}z = \partial_{t}z-2\partial_x z\partial_x\varphi-(\partial_x^2\varphi)z.
	\end{equation}
Let $Q=(0,T)\times (-\pi/2,\pi/2)$ and $dQ = dxdt$. Observe first that, $\cP_n^{+}z$ and $\cP_n^{-}z$ defined in \eqref{eq60} belong to $L^2(Q)$ by the definition of $D(\operatorname{M_{n}})$ and Lemma~\ref{lem::proprieties in D(M_n)}. So, developing the $L^2(Q)$ squared norm in identity \eqref{eq59},  leads to
\begin{equation}\label{eq61}
	\int_{Q}\cP_n^{+}z\cP_n^{-}zdQ\le\frac{1}{2}\int_{Q}\left|e^{-\varphi}\cP_ng\right|^{2}dQ.
\end{equation}

Straightforward computations using \cite[Lemmas 4.6 to 4.8]{tamekue2022null} lead to the following.

\begin{lemma}\label{lem::uniform Carleman scalar product}
	Let $n\in\N^{*}$. Then one has
    \begin{equation}\label{eq::scalar product}
\int_{Q}\cP_n^{+}z\cP_n^{-}zdQ=\int_{0}^{T}\int_{\operatorname{\omega}_{bdy}}\operatorname{Q_{bdy}}dQ+\int_{0}^T\int_{-b'}^{b'}\operatorname{Q_{con}}dQ+s\left(2n^2-\frac{1}{2}\right)\int_{0}^T\int_{-b'}^{b'}\theta(t)\frac{\beta'(x)\sin x}{\cos^3x}|z|^2dQ,
    \end{equation}
	where
	\begin{eqnarray}\label{eq::Q bdy}
		\operatorname{Q_{bdy}}&=&\frac{2s\theta}{\sin^2x}\left\{|\partial_x z|^2+|z|^2\right\}+\frac{2s^3\theta^3}{\sin^2x}\left(\frac{\cos x}{\sin x}+2\operatorname{sign}(x)\right)^2|z|^2\nonumber\\
		&&+\left\{2s^2\theta\theta'\left(\frac{\cos x}{\sin x}+2\operatorname{sign}(x)\right)^2-\frac{s\theta''}{2}(\ln|\sin x|+2x|+2)\right\}|z|^2\nonumber\\
		&&+\left\{-\frac{3s\theta}{\sin^4x}+\frac{s\theta(2n^2-1/2)}{\cos^2x}(1+2\operatorname{sign}(x)\tan x)\right\}|z|^2,
	\end{eqnarray}
	and
	\begin{equation}\label{eq::Q con}
		\operatorname{Q_{con}}=s\left\{\frac{\theta\beta^{(4)}}{2}-\frac {\theta''\beta}2+2s\theta\beta'^2(\theta'-s\theta^2\beta'')\right\}|z|^2-2s\theta\beta''|\partial_x z|^2.
	\end{equation}
\end{lemma}
\begin{remark}\label{rmk::key uniform}
	Thanks to \eqref{eq::key properties of beta in omega_con}, we can bound from below the third term in the right-hand side of \eqref{eq::scalar product} by $0$, for every $n\in\N^{*}$.
\end{remark}
Observe that $\operatorname{Q_{bdy}}$ is identical with $\operatorname{K_{bdy}}$ defined in \cite[eq. (4.23)]{tamekue2022null} with $A_1 = A_2=2$. Therefore, the same estimates provided in \cite[Lemma 4.11]{tamekue2022null} hold true for $\operatorname{Q_{bdy}}$, where positive constants involved in these inequalities depend only on $b'$.
\begin{lemma}\label{lem::uniform Carleman Q bdy}
	There exists a positive constant $C_0:=C_0(b')>0$ such that, for all
	\begin{equation}\label{eq::uniform Carleman s_2}
		s\ge C_0(T+T^2),
	\end{equation}
	the following inequality holds
	\begin{equation}\label{eq::uniform Carleman Q bdy}
		\int_{0}^{T}\int_{\operatorname{\omega}_{bdy}}\operatorname{Q_{bdy}}dQ\ge \int_{0}^{T}\int_{\operatorname{\omega}_{bdy}}2s\theta|\partial_x z|^2+2s\theta|z|^2+2s^3\theta^3|z|^2dQ.
	\end{equation}
\end{lemma}
\begin{lemma}\label{lem::uniform Carleman Q con}
	Assume \eqref{eq::uniform Carleman s_2}. Then there exist positive constants $C_1, C_2>0$ such that the following inequality holds
	\begin{equation}\label{eq::uniform Carleman Q con}
		\int_{0}^T\int_{-b'}^{b'}|\operatorname{Q_{con}}|dQ\le \int_{0}^T\int_{-b'}^{b'}C_1s\theta|\partial_x z|^2+C_2s^3\theta^3|z|^2dQ.
	\end{equation}
\end{lemma}
\begin{proof}
	Let $n\in\N^{*}$, then by \cite[Lemma 4.6]{tamekue2022null}, \eqref{eq::uniform Carleman s_2} and \eqref{eq::Q con} we have
	\begin{eqnarray*}
		|\operatorname{Q_{con}}|&\le&\left|\frac{s\theta\beta^{(4)}}{2}-\frac {s\theta''\beta}2+2s^2\theta\beta'^2(\theta'-s\theta^2\beta'')\right||z|^2+|2s\theta\beta''|| z|^2\nonumber\\
		&\le&C_1s\theta|\partial_x z|^2+C_2s^3\theta^3|z|^2,
	\end{eqnarray*}
	where $$C_1=C_1(b'):=2\max\{|\beta''(x)|: x\in[-b',b']\},$$
	$$C_2 =C_2(b') := \max\left\{\frac{\left|\beta^{(4)}(x)\right|}{32C_0^2}+2\frac{\left|\beta'(x)\right|}{C_0}+\frac{2\left|\beta'(x)\right|^2}{C_0}+2\beta'(x)^2|\beta''(x)|: x\in[-b',b']\right\}.$$
\end{proof}
The combination of Lemmas~\ref{lem::uniform Carleman Q bdy} and \ref{lem::uniform Carleman Q con} leads to the following.

\begin{lemma}\label{lem::intermediate uniform}	There exists positive constants $C_1:=C_1(b')$, $C_2:=C_2(b')$ and $C_0:=C_0(b')$ such that, for any $T>0$ and all
	$s\ge C_0(T+T^2)$, it holds
        \begin{multline}\label{eq::intermediate uniform}
	\int_{0}^{T}\int_{\operatorname{\omega}_{bdy}}2s\theta|\partial_x z|^2+2s\theta|z|^2+2s^3\theta^3|z|^2dQ\le\\
		\int_{0}^{T}\int_{\operatorname{\omega}_{con}}(C_1s\theta|\partial_x z|^2+C_2s^3\theta^3|z|^2)dQ
		+\frac{1}{2}\int_{Q}\left|e^{-\varphi}\cP_ng\right|^{2}dQ.
	\end{multline}
\end{lemma}

At this stage, to complete the proof of Carleman estimates \eqref{eq::uniform Carleman estimate}, it suffices to consider Lemma~\ref{lem::intermediate uniform} and to follow point by point~\cite[Lemmas 4.15 and 4.16]{tamekue2022null}.\\

\section{A technical lemma}\label{s:technical lemmas}
This section is devoted to the proof of a technical result used in the proof of Proposition~\ref{pro::proposition MM}

\begin{lemma}\label{lem:A-over-power}
Let $n\in\N^{*}$ and $\ell\in\N^{*}$ with $\ell\ge n$. Define
\begin{equation}\label{eq:A-elln-lb}
 C_{\ell,n}:=\frac{(2\ell+1)(n+1)}{2^{2n+2}}\frac{(\ell-n)!(\ell+n)!}{((\ell+1)!)^2}.
\end{equation}
Then, for $m:=\ell-n\ge 0$, one has
\begin{equation}\label{eq:lower bound Cln}
    C_{n+m,n}\ge\frac{(2n+2m+1)(n+1)}{4(2n+1)(n+m+1)^2(n+m)^m}. 
\end{equation}
\end{lemma}

\begin{proof}
First, one has
\begin{equation}\label{eq:1Cln}
    \frac{1}{C_{n+m,n}} = \frac{2^{2n+2}}{(2n+2m+1)(n+1)}R_{m,n}, 
\end{equation}
where
\[
R_{m,n}:=\frac{((n+m+1)!)^2}{m!(2n+m)!} = (n+m+1)^2\frac{((n+m)!)^2}{m!(2n+m)!}
\]
Since 
\[
\frac{((n+m)!)^2}{m!(2n+m)!} = \frac{\binom{n+m}{n}}{\binom{2n+m}{n}}
\]
and
\[
\binom{2n+m}{n}\ge\binom{2n}{n},\qquad\binom{n+m}{n}=\binom{n+m}{m}\le(n+m)^m,
\]
we obtain
\[
\frac{((n+m)!)^2}{m!(2n+m)!} \le\frac{(n+m)^m}{\binom{2n}{n}}
\]
Using the classical lower bound for the central binomial coefficient (use the Beta function $B(n+1,n+1)$ to obtain the estimate)
\[
\displaystyle\binom{2n}{n}\ge \frac{2^{2n}}{2n+1}
\]
we get
\begin{equation}\label{eq:Rmn}
    R_{m,n}:=(n+m+1)^2\frac{((n+m)!)^2}{m!(2n+m)!}\le \frac{(2n+1)(n+m+1)^2(n+m)^m}{2^{2n}}
\end{equation}
Combining~\eqref{eq:Rmn} with~\eqref{eq:1Cln} completes the proof of the statement.
\end{proof}

\printbibliography

\end{document}